\documentclass[a4paper,10pt]{article}
\usepackage{amsfonts,amsthm,amsmath,amssymb,amscd,graphicx,epsfig}
\usepackage{color}
\usepackage{subfigure}
\usepackage{xcolor}

\definecolor{Sepia}{rgb}{0.6,0.3,0.1}
\definecolor{RawSienna}{rgb}{0.2,0.7,0.1}
\definecolor{NavyBlue}{rgb}{0.2,0,0.5}
\definecolor{PineGreen}{rgb}{0,0.5,0}
\definecolor{Sepia1}{rgb}{0.9,0.0,0.1}
\definecolor{Kirpich}{rgb}{0.8,0.1,0.1}
\definecolor{Sepia2}{rgb}{0.4,0.,0.4}
\definecolor{Sepia3}{rgb}{0.4,0.,0.}
\definecolor{Sepia4}{rgb}{0.7,0.,0.7}
\definecolor{Sepia5}{rgb}{0.4,0.,0.7}
\theoremstyle{plain}

\newtheorem{lemma}{Lemma}

\theoremstyle{definition}

\newtheorem{remark}{Remark}

\title{On local and global aspects of the 1:4 resonance in the conservative cubic H\'enon maps.
}
\author{M. Gonchenko$^1$, S.V. Gonchenko$^2$, I. Ovsyannikov$^{2,3}$ and A. Vieiro$^1$\\
{\small $^1$ Departament de Mathem\`atiques i Inform\`atica, Universitat de Barcelona, Spain
}\\
{\small $^2$  N.I.~Lobachevsky Nizhny Novgorod
 University, Russia}\\
{\small $^3$  Fachbereich Mathematik und Informatik, Universit\"at Bremen, Germany}\\
{\small { \tt  gonchenko@ub.edu, \tt gonchenko@pochta.ru, \tt ivan.i.ovsyannikov@gmail.com, \tt vieiro@maia.ub.es}}
}
\date{}

\begin{document}

\maketitle

%\tableofcontents

\begin{abstract}
We study the 1:4 resonance for the conservative cubic H\'enon maps $\mathbf{C}_\pm$ 
with positive and negative cubic term. These maps show up different bifurcation
structures both for fixed points with eigenvalues $\pm i$ and for
4-periodic orbits.  While for $\mathbf{C}_-$ the 1:4 resonance unfolding has
the so-called Arnold degeneracy (the first Birkhoff twist coefficient  equals
(in absolute value) to the first resonant term coefficient), the map $\mathbf{C}_+$
has a different type of degeneracy because the resonant term can vanish. In the
last case, non-symmetric points are created and destroyed at pitchfork
bifurcations and, as a result of global bifurcations, the 1:4 resonant chain of
islands rotates by $\pi/4$. For both maps several bifurcations are detected and illustrated.
\end{abstract}

{\bf Keywords:} strong 1:4 resonance, cubic H\'enon map, bifurcations, 4-periodic orbits.

{\bf Mathematical Subject Classification:}  37G10, 37G20, 37G40, 37J10, 37J20.

\section{Introduction}

We study the 1:4 resonance problem for the conservative cubic H\'enon maps 
\begin{equation}
\mathbf{C}_-: \bar x = y, \;\;\; \bar y= M_1 -  x + M_2 y  - y^3
\label{cubH1m}
\end{equation}
and
\begin{equation}
\mathbf{C}_+: \bar x = y, \;\;\; \bar y= M_1 -  x + M_2 y  + y^3,
\label{cubH1pl}
\end{equation}
where $x, y$ are coordinates and $M_1, M_2$ are parameters.

For area-preserving maps, the basis of the 1:4 resonance phenomenon consists in 
a local bifurcation of a fixed (or periodic) point with eigenvalues
$e^{\pm i \pi/2} = \pm i$. However, this can be only the simplest, standard
part of the general picture of the resonance. As we show, the  1:4 resonance in the
case of maps (\ref{cubH1m}) and (\ref{cubH1pl}) can be nontrivial, i.e. it can
include not only bifurcations of fixed points with eigenvalues $\pm i $
themselves (local aspects) but also a series of accompanying bifurcations
(global aspects) of 4-periodic orbits which are initially born from the
central fixed point with eigenvalues $\pm i $.

It is well-known that the strong resonances, i.e. the bifurcation phenomena
connected with the existence of fixed (periodic) points with eigenvalues  $e^{\pm
2\pi i/q}$ with $q=1,2,3,4$ (that is, the 1:1, 1:2, 1:3 and 1:4 resonances),
play a very important role in the dynamics of area-preserving
maps. Among them, the 1:4 resonance (with $q=4$) is, as a rule, the most
complicated and least studied. For example, the resonances with $q=1,2,3$
are nondegenerate for the standard conservative H\'enon map
\begin{equation}
\bar x = y, \;\;\; \bar y= M -  x - y^2.
\label{H2st}
\end{equation}
Unlike this, the resonance 1:4 is degenerate here \cite{Bir87,SV09}.

The Birkhoff local normal form at the fixed point with multipliers $\pm i$,
expressed in complex coordinates $z= x+iy$ and $z^* = x-iy$, can be written as follows 
\begin{equation}
\bar z = iz +  (B_{21}|z|^2 + B_{32} |z|^4) z  + B_{03} (z^*)^3 +  B_{50} z^5 +
 B_{14} z (z^*)^4 + O(|z|^7),
\label{nf14}
\end{equation}
where, a priori, the coefficients $B_{ij}$ are complex. The area-preserving
property requires that (i) $\text{Im}(B_{21})=0$, (ii)~$-3 B_{03}^* B_{21} - 5
i B_{50} +i B_{14}^* = 0$, and (iii) $3 B_{21}^2 + 6 \text{Im}(B_{32}) - 9
|B_{03}|^2 = 0$. Moreover, rotating the complex coordinates one can consider
$B_{03}$ real. Condition (ii) implies in this case that $\text{Re}(B_{14}) = 5
\text{Re}(B_{50})$. The main nondegeneracy conditions of the normal form (\ref{nf14}) are
as follows
\begin{equation} \label{cond}
B_{03}\neq 0 \;\;\mbox{and} \;\; A = \left|\frac {B_{21}}{B_{03}}\right| \neq  1.
\end{equation}
The simplest degeneracies occur when only one of the conditions in (\ref{cond}) does
not hold.

The local bifurcations at the 1:4 resonance, in the general (not necessarily
conservative) setting, were first studied by V.~Arnold in the 70's, see
\cite{Arn-Geom}, for the flow normal form $\dot z = \varepsilon z +
\tilde{A} z |z|^2 + (z^*)^3$, 
where $\tilde A$ is a coefficient and $\varepsilon$ is a small complex parameter.
\footnote{However, one can consider that $\varepsilon$ varies inside the unit
circle $|\varepsilon| =1$ (as it was done e.g. in \cite{AAIS86}) by rescaling
$t =T/|\varepsilon|, z = Z |\varepsilon|$.}
He showed that the structure of such resonance essentially depends
on the relation between
$\mbox{Re}\;\tilde{A}$ and $\mbox{Im}\;\tilde{A}$ and he studied several cases,
e.g. when $|\tilde{A}|<1$ or $|\mbox{Re}\;\tilde{A}|>1$. Numerous
other cases (when $|\mbox{Re}\;\tilde{A}| <1$ and $|\tilde{A}|>1$) were studied
in many papers, see e.g. \cite{Kr94,Kuz}. Concerning the conservative case,
where $\mbox{Re}\;\tilde{A}\equiv 0$, the Arnold normal form can be represented as
\begin{equation} \label{Arnold_conservative}
 \dot z = i \varepsilon z + i b z |z|^2 + i (z^*)^3, 
\end{equation}
where $\varepsilon$ and $b= \text{Im}(\tilde{A})$ are real. This normal form is nondegenerate
in the case $b \neq 1$ and its bifurcations are well-known, see
Fig.~\ref{nondegcub}. 
We see that cases $b<-1$, Fig.~\ref{nondegcub}a, and $|b|<1$,
Fig.~\ref{nondegcub}b, are very different. In particular, the equilibrium $z=0$
is always stable in the first case, whereas, it can be unstable (a saddle with
8 separatrices, at $\varepsilon =0$) in the second case.

\begin{figure}[htb]
\begin{center}
\includegraphics[width=14cm]{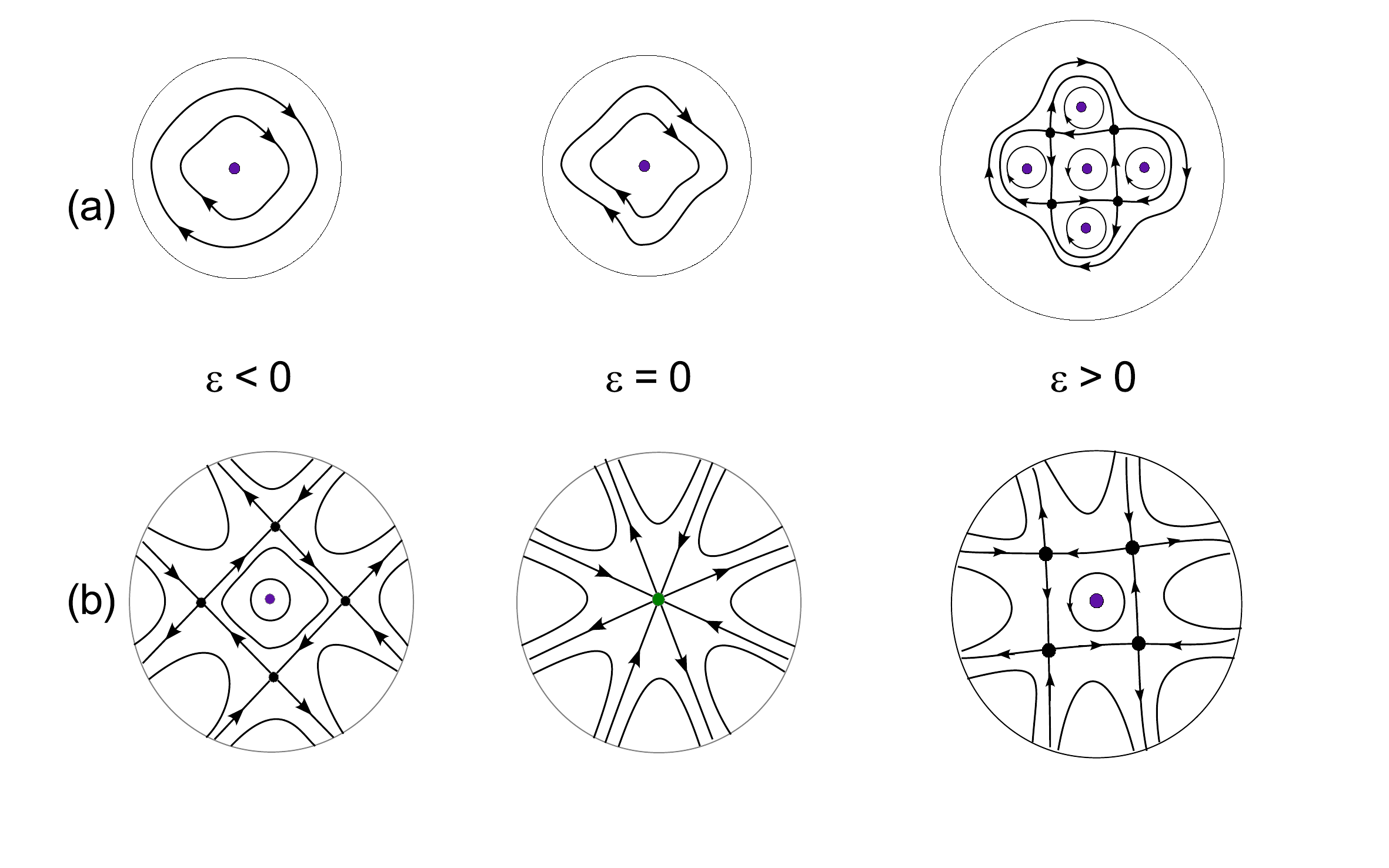}
\end{center}
\caption{Bifurcations of zero equilibrium in the family (\ref{Arnold_conservative})
for $|b| \neq 1$ in the cases (a)
$b < -1$, here only one equilibrium $z=0$ (a nonlinear center) exists at
$\varepsilon \leq 0$  and 8 equilibria appear surrounding $z=0$ at
$\varepsilon > 0$; and (b) $|b|<1$, here 5 equilibria (4 saddles and the
center $z=0$) exist at $\varepsilon<0$; at $\varepsilon=0$ all these equilibria
merge to the point $z=0$ which becomes the nonhyperbolic saddle with 8
separatrices; at $\varepsilon>0$ the point $z=0$ becomes again a center and the 4
saddles appear to be rotated by $\pi/4$ with respect to the
ones at $\varepsilon<0$.  Note that, in case $b>1$, one needs to
change both $\varepsilon$ by $-\varepsilon$ and the time direction in the (a)
row of the plot.}
\label{nondegcub}
\end{figure}

As it was shown in \cite{Bir87,SV09}, in the case of map (\ref{H2st}), the
degeneracy $A=1$ (here $B_{21} = - B_{03}$) takes place.
Note that the H\'enon map (\ref{H2st}) has a fixed point with eigenvalues $\pm
i$ at $M=0$. As it was shown in \cite{Bir87}, the family $\bar x = y, \;\;\; \bar y=
\varepsilon_1  -  x - y^2 + \varepsilon_2 y^3$ can be considered as a
two-parameter general unfolding for the study of bifurcations of this point.
This result is quite important, since such maps (conservative H\'enon maps with 
small cubic term) appear naturally as rescaled first return maps near
homoclinic orbits to saddle-focus equilibria of divergence-free
three-dimensional flows \cite{BS89} or near quadratic homoclinic tangencies of
area-preserving maps \cite{GG09,DGG15}.\footnote{ The main bifurcations of
area-preserving maps with quadratic homoclinic tangencies were studied in
\cite{MR97,GG09,DGG15}, and with cubic ones in \cite{GGO17}. In all these
papers, the main technical tool was the so-called rescaling method by which it
was possible to represent the first return map in the form of a map being
asymptotically close to the quadratic or to the cubic H\'enon map.}

In the present paper we show that degeneracy $A =1$ can take place only in
the case of cubic map (\ref{cubH1m}) (here $B_{03}<0$ always). This occurs for
$M_1 =\pm 16/27$ and $M_2 =1/3$.  For those parameter values one has
$B_{21}\neq 0$ and $B_{03}\neq 0$, see Section~\ref{sec:map_mns} for the
concrete expressions of these coefficients as a function of $M_2$. Then,
according to \cite{Bir87}, case $A=1$ is generic whenever the coefficient
$\Omega=\text{Re}(B_{32}-B_{50}-B_{14})$ is non-zero, which guarantees a non-vanishing
twist for the 4th power of the local normal form (\ref{nf14}). We
show in Fig.~\ref{Omeg} the graph of $\Omega$ for $\mathbf{C}_{-}$ and
$\mathbf{C}_{+}$ as a function of $M_2$.  In particular, we see that $\Omega
\approx -0.25$ when degeneracy $A=1$ takes place for
$\mathbf{C}_-$. A description of the local bifurcations in 
normal form (\ref{nf14}) with $A=1$, $\Omega \neq 0$ was also given in
\cite{Bir87}.  However, we do not restrict ourselves only to the study of the
local structure of this resonance -- the analysis of the corresponding normal
form is quite standard.  We also study the global effect of this resonance on
the dynamics of map (\ref{cubH1m}) as a whole. To this end, we find (analytically
and/or numerically) bifurcation curves relevant to describe the bifurcation
diagrams related to both bifurcations of the fixed point with eigenvalues
$\pm i$ as well as to accompanying bifurcations of saddle and elliptic
4-periodic orbits belonging to the corresponding resonant garland surrounding
the fixed point. We collect the corresponding results in
Section~\ref{sec:map_mns}.

\begin{figure}[tb]
\begin{center}
\includegraphics[width=0.48\textwidth]{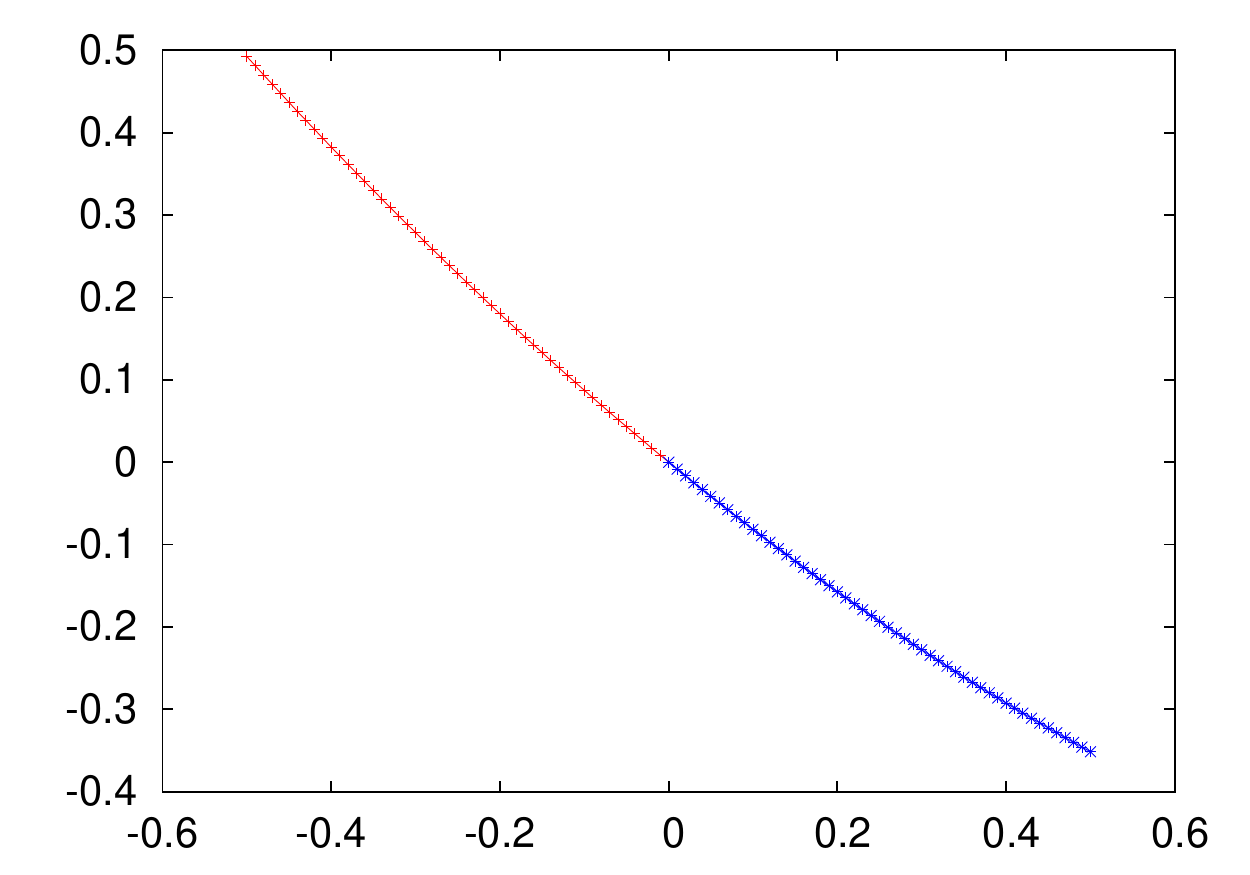}
\caption{We represent the value of the Birkhoff twist coefficient $\Omega$ of
$\mathbf{C}_{+}^4$ (for $M_2<0$) and of $\mathbf{C}_{-}^4$ (for $M_2>0$) as a
function of $M_2$.}
\label{Omeg}
\end{center}
\end{figure}

We also show that degeneracy $B_{03}=0$ can take place only in the case of
the cubic map (\ref{cubH1pl}), when $M_1 =\pm 20/27$ and $M_2 = -1/3$ (for this
map $B_{21}\neq 0$ always).  Moreover, for map $\mathbf{C}_+$, the value of
$A$ is always greater than 1. As in the case of map (\ref{cubH1m}), we study
both the local and global aspects of this resonance, see
Section~\ref{sec:map_pl}. 

As far as we know, this type of the conservative 1:4 resonance (with
$B_{03}=0$) has not been studied before,\footnote{However, it was noted in
\cite{GLRT14,GT17} that such type resonances can provoke symmetry-breaking
bifurcations (of pitchfork type) which, in the case of reversible maps, lead
to the appearance of nonconservative periodic orbits (e.g. periodic sinks and
sources).}  therefore we describe the main elements of the
local bifurcations at such resonance in the Appendix~\ref{appendix1p4deg}. In this case we assume $B_{14}^* = 5
B_{50} \neq 0$. Note that the previous equality of the coefficients guarantees the Hamiltonian character of the
corresponding flow normal form (see Eq.(\ref{2p5_3}) in Section~\ref{sec:14apm}), while the nonvanishing of those
coefficients is an additional nondegeneracy condition. In Fig.~\ref{B14-B50p}
we show the graphs of $B_{14}$ and $B_{50}$, for the case of map
(\ref{cubH1pl}), as a function of $M_2$. In particular, when
degeneracy $B_{03}=0$ takes place in map (\ref{cubH1pl}) ($M_1=\pm 20/27,
M_2 = -1/3$), one has that $B_{14} \approx -5/64$.

\begin{figure}[tb]
\begin{center}
\includegraphics[width=0.48\textwidth]{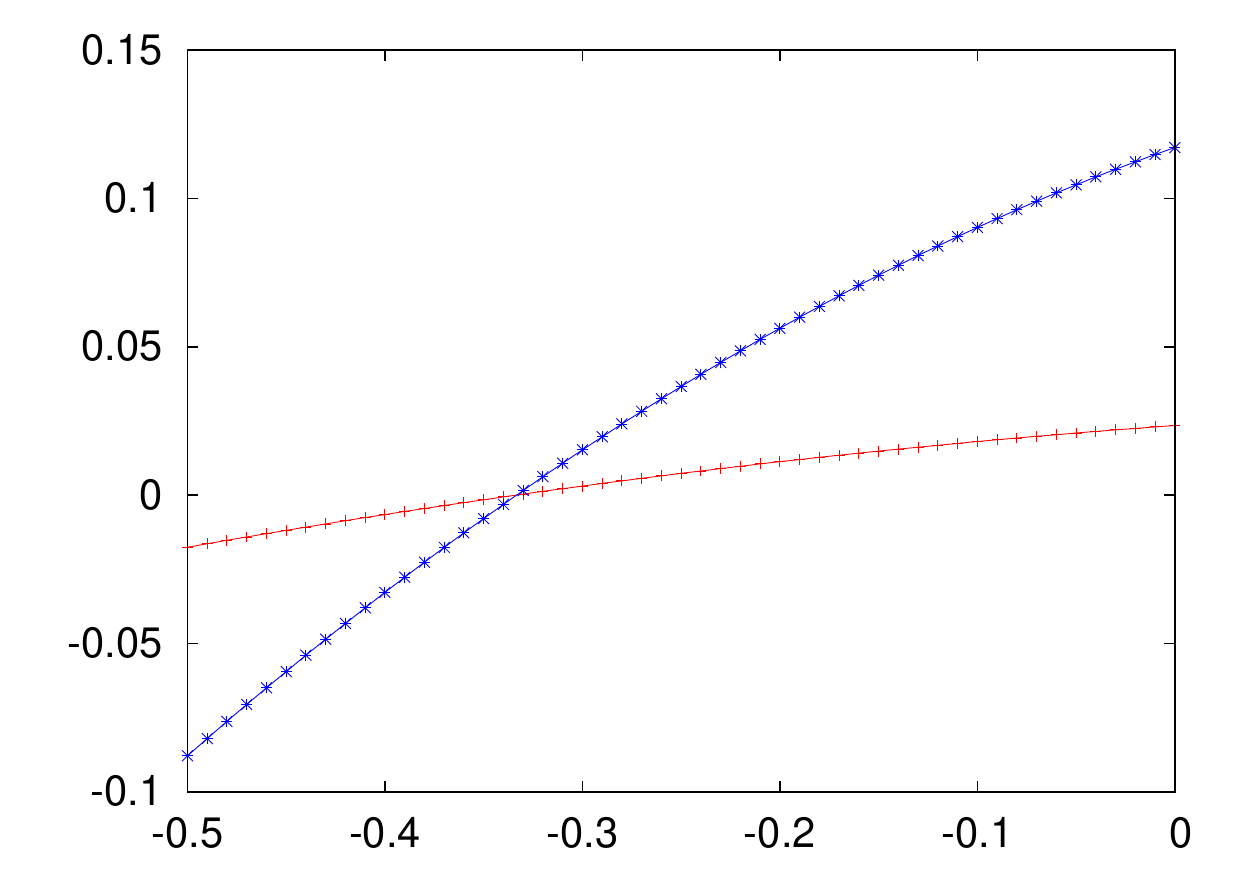}
\includegraphics[width=0.48\textwidth]{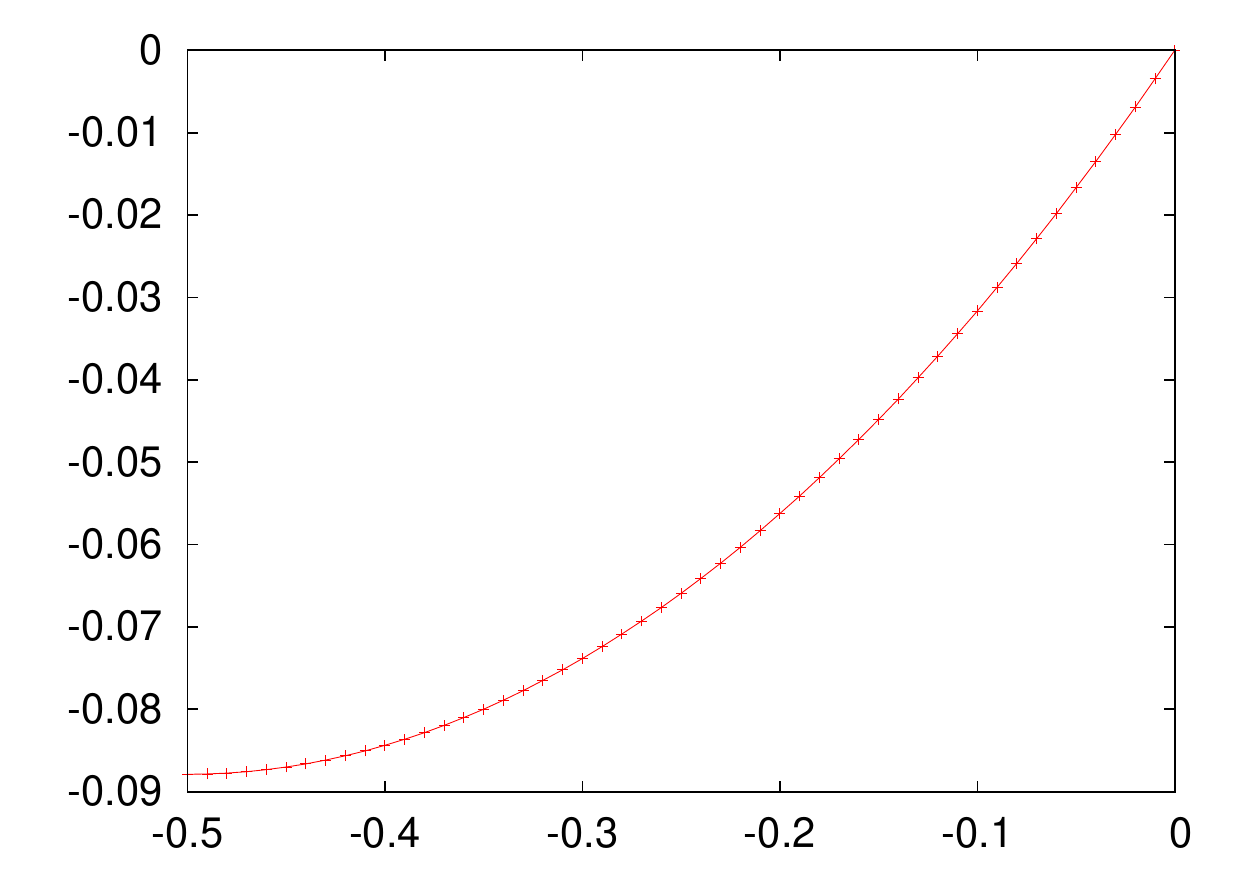}
\caption{For map $\mathbf{C}_+$, we represent in the left plot the values
of  $\text{Im}(B_{14})$ (in red) and of $5 \text{Im}(B_{50})$ (in blue) as a function of $M_2$.  Both
lines cross at $0$ for the value $M_2 = -1/3$ where degeneracy $B_{03}=0$
takes place. In the right plot we represent  $\text{Re}(B_{14})=5
\text{Re}(B_{50})$ as a function of $M_2$ for the same map.
}
\label{B14-B50p}
\end{center}
\end{figure}

We note that the cubic H\'enon maps (\ref{cubH1m})-(\ref{cubH1pl}) have an important meaning
for the theory of dynamical systems. For example, they appear as truncated
normal forms of first return maps near cubic homoclinic tangencies. In
Fig.~\ref{2typcubt} we illustrate the geometric idea how such maps can be
obtained. Let a two-dimensional map $f$ have a fixed saddle point $O$ and a
homoclinic orbit $\Gamma$ at whose points the manifolds $W^u(O)$ and $W^s(O)$
have a cubic tangency. Let $M^+ \in W^s_{loc}$ and $M^- \in W^u_{loc}$ be a
pair of points of $\Gamma$ and $\sigma_k$ be a small (horizontal) strip near
$M^+$. Under some number $m$ of iterations of $f$ the strip $\sigma_k$ is
mapped into a vertical strip $\tilde\sigma_k$ located near the point $M^-$. The
(local) map from $\sigma_k$ into $\tilde\sigma_k$  can be represented, for
simplicity, as the linear map $\bar x = \lambda^m x, \bar y = \lambda^{-m}y$,
where $(x,y)$ are coordinates of points in  $\sigma_k$ and $(\bar x,\bar y)$
are those in $\tilde\sigma_k$. Let $q$ be an integer such that $f^q(M^-) =
M^+$. Then, since the curve $f^q(W^u_{loc})$ have a cubic tangency with
$W^s_{loc}$ at the point $M^+$, the image $f^q(\tilde\sigma_k)$ of
$\tilde\sigma_k$ will have the form of a cubic horseshoe. Thus, the geometry of
the first return map $f^k : \sigma_k \to \sigma_k $, where $k=m+q$, is like the
one of a cubic horseshoe map. If one rescales  the initial coordinates $(x,y)$
and the initial parameters $\mu_1$ and $\mu_2$ (that are the usual parameters that 
unfold the initial cubic homoclinic tangency between the curves
$f^q(W^u_{loc})$ and $W^s_{loc}$ at the point $M^+$) in the appropriate way,
then one can rewrite the map $f^k : \sigma_k \to \sigma_k $ in the form of a
cubic H\'enon maps with some terms that are asymptotically small as
$k\to\infty$. Note that if $\lambda>0$, then there are two different types of
cubic homoclinic tangencies: the tangency ``incoming from above'',
see Fig.~\ref{2typcubt}a, and the tangency ``incoming from below'', see
Fig.~\ref{2typcubt}b. In the first case, the truncated rescaled map is map
(\ref{cubH1m}), while in the second case it is map (\ref{cubH1pl}). For
more details, see \cite{GGO17} for the area-preserving case and
\cite{GST96,GSV13} for the dissipative case.

\begin{figure}[htb]
\centerline{
  \includegraphics[width=14cm]{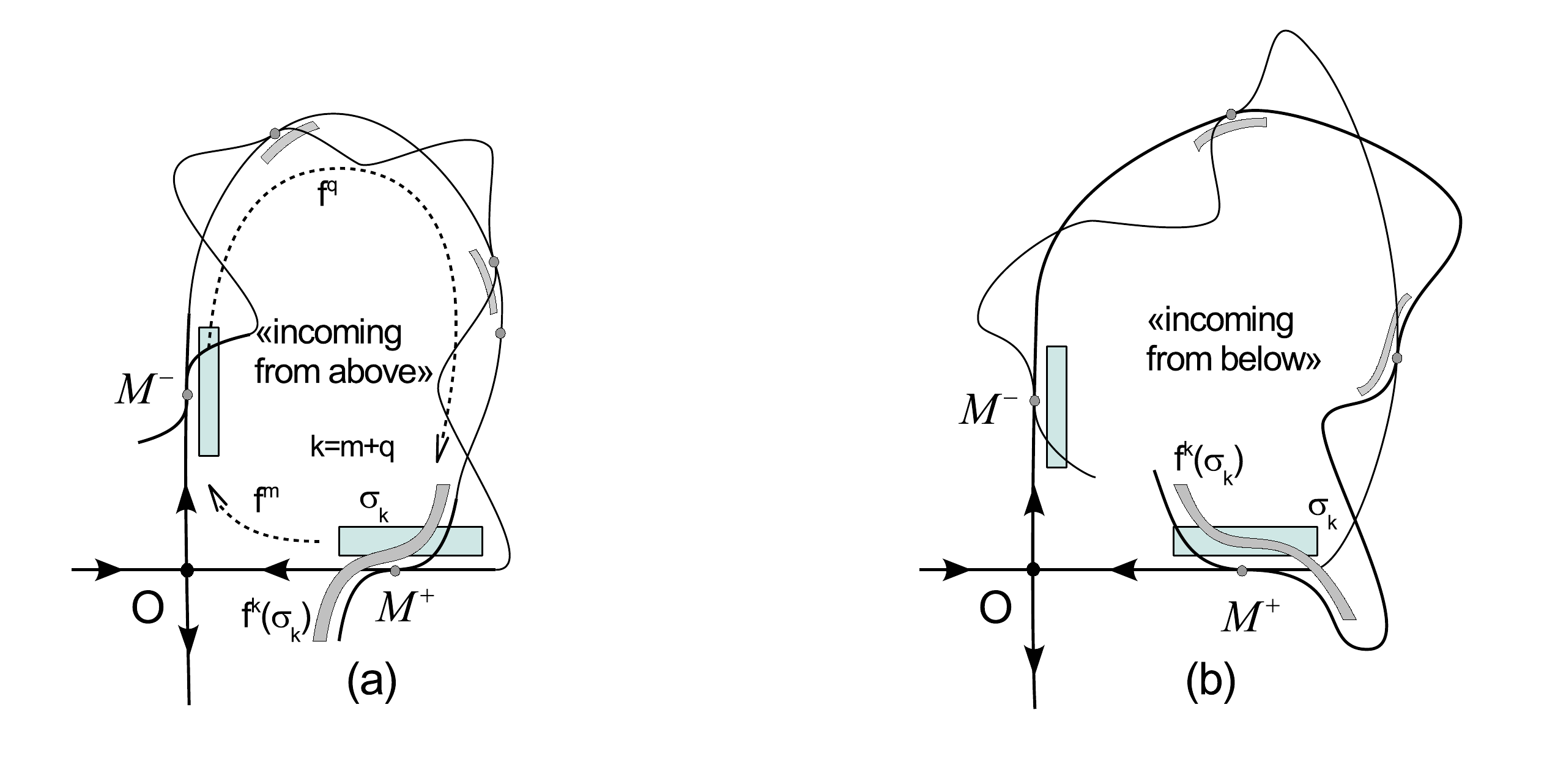}}
  \caption{Two types of cubic homoclinic tangencies to a saddle fixed point
   with positive eigenvalues and geometry of the first return maps. }
  \label{2typcubt}
\end{figure}

The main bifurcations of the dissipative cubic H\'enon maps (when the absolute
value of the Jacobian is less than~1) were studied in \cite{Gon85,GK88}.  In
the paper \cite{DM00} some bifurcations of conservative maps (\ref{cubH1m}) and
(\ref{cubH1pl}) were studied.  Note that one of the main goals of \cite{DM00}
was to describe the bifurcation structure of the 1:1, 1:2 and 1:3 strong
resonances. The present work as well as the papers \cite{MGon05,GGO17} are
devoted mainly to the study of the local and global aspects of the 1:4
resonances, which complements the research of \cite{DM00} on the dynamics and
the bifurcations of the conservative cubic H\'enon maps.

The paper is organised as follows. In Section~\ref{sec:14apm} we review the 1:4
resonance for general area-preserving maps and comment about the degenerate
cases. In Section~\ref{sec:map_mns} we consider map $\mathbf{C}_-$
and describe the local and global aspects of the 1:4 resonance
in this concrete case.  The degeneracy $A=1$ occurs here and we study its influence
on the global bifurcation diagram.  The same type of local and global bifurcation
analysis is performed in Section~\ref{sec:map_pl} for the 1:4 resonance of the
map $\mathbf{C}_+$. We show that degeneracy $B_{03}=0$ takes place in
this case. As far as we know, this degeneracy has not been studied before, therefore we include
the normal form analysis of this bifurcation in Appendix~\ref{appendix1p4deg}.
For both maps, some of the bifurcation curves have been explicitly obtained and their
equations are derived in Appendix~\ref{sec:4par}. 
Finally, in Section~\ref{sec:conclusions} we comment on related topics where
the results obtained in this paper could be relevant.

\section{Local aspects of the 1:4 resonance in area-preserving maps.} \label{sec:14apm}

The unfolding of the non-degenerate 1:4 resonance leads to the one-parameter
family of area-preserving maps 
\begin{equation}
\begin{array}{l}
\bar z=i e^{i \tilde{\beta}} z + B_{21} z^2 z^* +
          B_{03}(z^{*})^3+O(|z|^5),
\label{HeComplNew13}
\end{array}
\end{equation}
being $\tilde{\beta}$ a real parameter characterizing the deviation of the angle argument
$\varphi$ of the eigenvalues of the fixed point from~$\pi/2$ ($\varphi = \tilde{\beta} +
\pi/2$), the coefficients $B_{21} :=B_{21}(\tilde{\beta})$ and $B_{03}:=B_{03}(\tilde{\beta})$
are real and depend smoothly on~$\tilde{\beta}$.
If $B_{03} \neq 0$, the fourth iteration of map (\ref{HeComplNew13}) can be locally embedded
into the one-parameter family (\ref{Arnold_conservative}) of Hamiltonian flows, being $b(\tilde \beta)
= B_{21}(\tilde \beta)/B_{03}(\tilde \beta)$ and $\varepsilon = 4\tilde\beta/B_{03}$. Local bifurcations of
this Hamiltonian system are shown in Fig.~\ref{nondegcub}.

As mentioned in the introduction, the 1:4 resonance is degenerate whether
$A=|b(0)|=1$ or $B_{03}(0) =0$, see (\ref{cond}).

The case $A=1$ appears in the bifurcation diagram of $\mathbf{C}_-$, see
Section~\ref{sec:map_mns}.  The unfolding of this case  leads to a
two-parameter family of area-preserving maps.  The fourth iteration of such a
family is close-to-identity and can be approximated by the flow of the
Hamiltonian system 
\begin{equation}
\displaystyle
\dot z =  i \beta z +
i(1+\mu) z |z|^2 + i z^{*3}  + i \tilde{B}_{32}|z|^4z   +  i \tilde{B}_{50} z^5 +  i \tilde{B}_{14} |z|^2 z^{*3} + O(|z|^7),
\label{2p5_2}
\end{equation}
where  $\beta = 4 \tilde{\beta}/ B_{03}$, 
$\mu$ is the parameter responsible for the deviation from $A=1$.
The coefficients $\tilde{B}_{ij}$ are related to those in (\ref{nf14}), see \cite{Bir87,Gelf,SV09}.
Namely, one has
 \[1+\mu = \frac{B_{21}}{B_{03}} + \mathcal{O}(\beta), \quad  \tilde{B}_{32} = \frac{\text{Re}(B_{32})}{B_{03}} + \mathcal{O}(\beta), \quad
 \tilde{B}_{50} = \frac{\text{Re}(B_{50})}{B_{03}} + \mathcal{O}(\beta), \quad \tilde{B}_{14} = \frac{\text{Re}(B_{14})}{B_{03}} + \mathcal{O}(\beta).
\]
The vector field has zero divergence provided that $5 \tilde{B}_{50}=\tilde{B}_{14}$.

The bifurcation diagram for~(\ref{2p5_2}), when considering $(\beta,\mu)$ in a
small neighbourhood of the origin (case $A=1$), is displayed in
Fig.~\ref{fig_pi2-A1}.  There are three bifurcation curves $L_1$, $L_2$ and
$L_3$ dividing the~$(\beta,\mu)$-parameter plane into 3 domains. Curves
$L_1:\{\beta=0,\mu>0\}$ and  $L_2:\{\beta=0,\mu<0\}$ correspond to the passage
from I to II and II to III, respectively, and reconstructions of nonzero saddle
equilibria occur. Curve~$L_3:\{\mu=\sqrt{-\beta \Omega} + O(|\beta|)\}$
corresponds to a $\pi/2$-equivariant parabolic bifurcation: 4 saddles and 4
elliptic equilibria appear when crossing from I to III.
This bifurcation takes place far away from the origin of coordinates and it is
a codimension one bifurcation for~(\ref{2p5_2}) since the flow is invariant
under the rotation of angle~$\pi/2$. 
Note that the origin is a non-degenerate conservative center
for~$\beta \neq 0$; it is a degenerate saddle with 8 separatrices
for~$(\beta,\mu)\in L_1$ and a degenerate conservative center for
$(\beta,\mu)\in L_2$. The origin in the~$(\beta,\mu)$-plane, which corresponds
to the case $A=1$, is the endpoint of all the three bifurcation curves. 
Note that, in Fig.~\ref{fig_pi2-A1}, we represent the bifurcation diagram for the
case $\Omega <0$ because for map $\mathbf{C}_{-}$, as we have computed, one
has $\Omega \approx -0.25$, see Fig.~\ref{Omeg}.

\begin{figure}[tb]
\centering
\includegraphics[height=8cm]{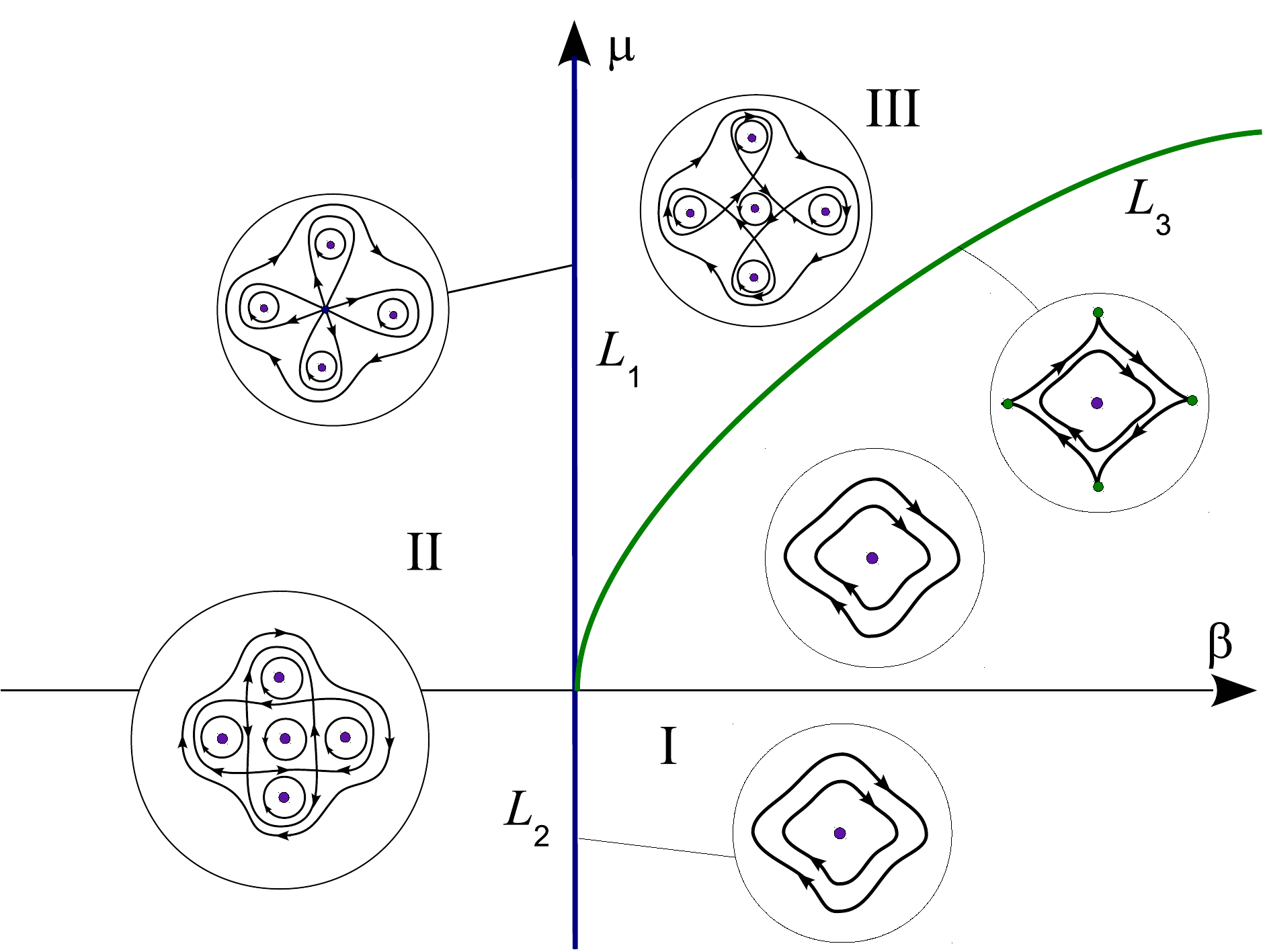}
\caption{ Main elements of bifurcation diagram in the case $A=1$ (we assume $\Omega <0$).}
\label{fig_pi2-A1}
\end{figure}

The unfolding of case $B_{03}=0$ reduces to the study of the Hamiltonian flow
\begin{equation}
\displaystyle
\dot z =  i \hat{\beta} z + i \hat{B}_{21} z |z|^2 + i \epsilon z^{*3}  + i \hat{B}_{32}|z|^4z   +  i \hat{B}_{50} z^5 +  i \hat{B}_{14} |z|^2 z^{*3} + O(|z|^7),
\label{2p5_3}
\end{equation}
where $\hat{\beta}=4\tilde{\beta}$ and $\epsilon$ are real parameters and $\hat{B}_{ij} = \text{Re}(B_{ij}) + \mathcal{O}(\hat{\beta})$ comparing with (\ref{nf14}).
This degeneracy appears only in the case of map $\mathbf{C}_{+}$, see Section~\ref{sec:map_pl}.  
The generic scenario happens whenever $B_{21}\neq 0$ and $\text{Re}(B_{50}) \neq 0$. 
The bifurcation diagram for (\ref{2p5_3}), when considering
$(\epsilon,\hat{\beta})$ in a small neighbourhood of the origin, is displayed in
Fig.~\ref{fig_pi2-B03}. 
The curve $L_2 : \{ \hat{\beta} =0 \}$ corresponds to
the 1:4 resonance of the fixed point (the equilibrium $z=0$ of flow (\ref{2p5_3}) has two zero eigenvalues). The curves $L^{\pm} : \{ \hat{\beta} =
\frac{\epsilon}{2 \text{Re}(B_{50})} \left( B_{21}  \pm
\frac{\epsilon}{2} \right) + \mathcal{O}(\hat{\beta} \epsilon), \ \epsilon \,
\text{Re}(B_{50})<0 \}$ correspond to $\pi/2$-equivariant pitchfork
bifurcations related to the creation of 8 nonzero equilibria of the system
(\ref{2p5_3}) (4 of the equilibria correspond to a saddle 4-periodic orbit of the map while the other 4 equilibria correspond to an elliptic 4-periodic orbit of the map).  The illustration in
Fig.~\ref{fig_pi2-B03} corresponds to $\text{Re}(B_{50})<0$ and $B_{21}>0$, as
happens for the map~$\mathbf{C}_+$ for $M_2=-1/3$ when degeneracy $B_{03}=0$ takes
place (see Fig.~\ref{B14-B50p} right and Section~\ref{sec:map_pl}). 

Note that normal forms (\ref{2p5_2}) and (\ref{2p5_3}) are Hamiltonian and
reversible with respect to two linear involutions: involution $R:\; z\to
z^*$ (in the real coordinates $(x,y)$, it corresponds to involution $(x,y)\to
(x,-y)$) and involution $R^*: z \to i z^*$ (it corresponds to 
involution $(x,y)\to (y,x)$). Since flows (\ref{2p5_2}) and (\ref{2p5_3}) are
$\pi/2$-equivariant, two additional involutions also exist: $\tilde R:\; z\to
-z^*$ (it corresponds to $(x,y)\to (-x,y)$) and $\tilde R^*: z \to -i z^*$
(it corresponds to the involution $(x,y)\to (-y,-x)$). The lines of fixed
points of these involutions are the following: $\text{Fix}(R)= \{y=0\}$; $\text{Fix}(\tilde R)=
\{x=0\}$; $\text{Fix}(R^*)= \{x=y\}$ and $\text{Fix}(\tilde R^*)= \{x=-y\}$.

Returning to the bifurcation diagram of Fig.~\ref{fig_pi2-B03}, we see that at
$(\hat{\beta},\epsilon)\in \text{I}$ only one equilibrium exists (the trivial equilibrium
$z=0$ that is a center), while nontrivial equilibria (centers and saddles)
appear in the other regions of the diagram. In regions II and IV there are 8
nontrivial equilibria while in region III there exist 16 nontrivial equilibria.  In
the case of regions II and IV, all nontrivial equilibria are symmetric, i.e.
belong to the lines of fixed points of the involutions: in II two of the four
centers belong to the axis $y=0$ ($\text{Fix}(R)$) while the other two belong
to $x=0$ ($\text{Fix}(\tilde R)$), and two of the four saddles belong to the
bisectrix $x=y$ ($\text{Fix}(R^*)$) while the other two belong to $x=-y$
($\text{Fix}(\tilde R^*)$). We see in region IV another disposition of these
equilibria: the picture seems turned by an angle of $\pi/4$. Due
to the strong reversibility properties of system (\ref{2p5_3}), such simple
rotation of the garland is impossible without bifurcations. The corresponding
(providing such a rotation) symmetry breaking bifurcations, at the passage from
domain II to domain IV, are schematically shown in
Fig.~\ref{fig_pi2-B03}. When passing from I to III$_b$ the centers undergo
(supercritical) pitchfork bifurcations: they all become saddles and four pairs
of nonsymmetric centers are born. The curve $L_{\text{hom}}$ in domain III
corresponds to a global $\pi/2$-equivariant bifurcation of creation of
heteroclinic connections between all 8 saddles.\footnote{For the case of map
$\mathbf{C}_+$, a homoclinic zone should exist instead of the simple curve
$L_{\text{hom}}$.}  The passage through $L_{\text{hom}}$ from  III$_b$ to
III$_a$ is related to reconstructions of the separatrices of the saddles. After
this, at passage from III$_a$ to IV, we observe a (subcritical) pitchfork
bifurcation where asymmetric centers merge with symmetric saddles and the
latter become centers.

Some details of
the analysis of normal form  (\ref{2p5_3}) are given in Appendix~\ref{appendix1p4deg} \footnote{
In Appendix~\ref{appendix1p4deg} we denote $b_1 = \hat{\beta}$, $\mu=2 \epsilon$, $b_2 = 2 \hat{B}_{21}$
and $B=16 \hat{B}_{50}$.}.  See Fig.~\ref{fig:res14_p_M20p5} for the bifurcation
scenario taking place for $\mathbf{C}_{+}$ when crossing the region III.

\begin{figure}[tb]
\centering
\includegraphics[height=8cm]{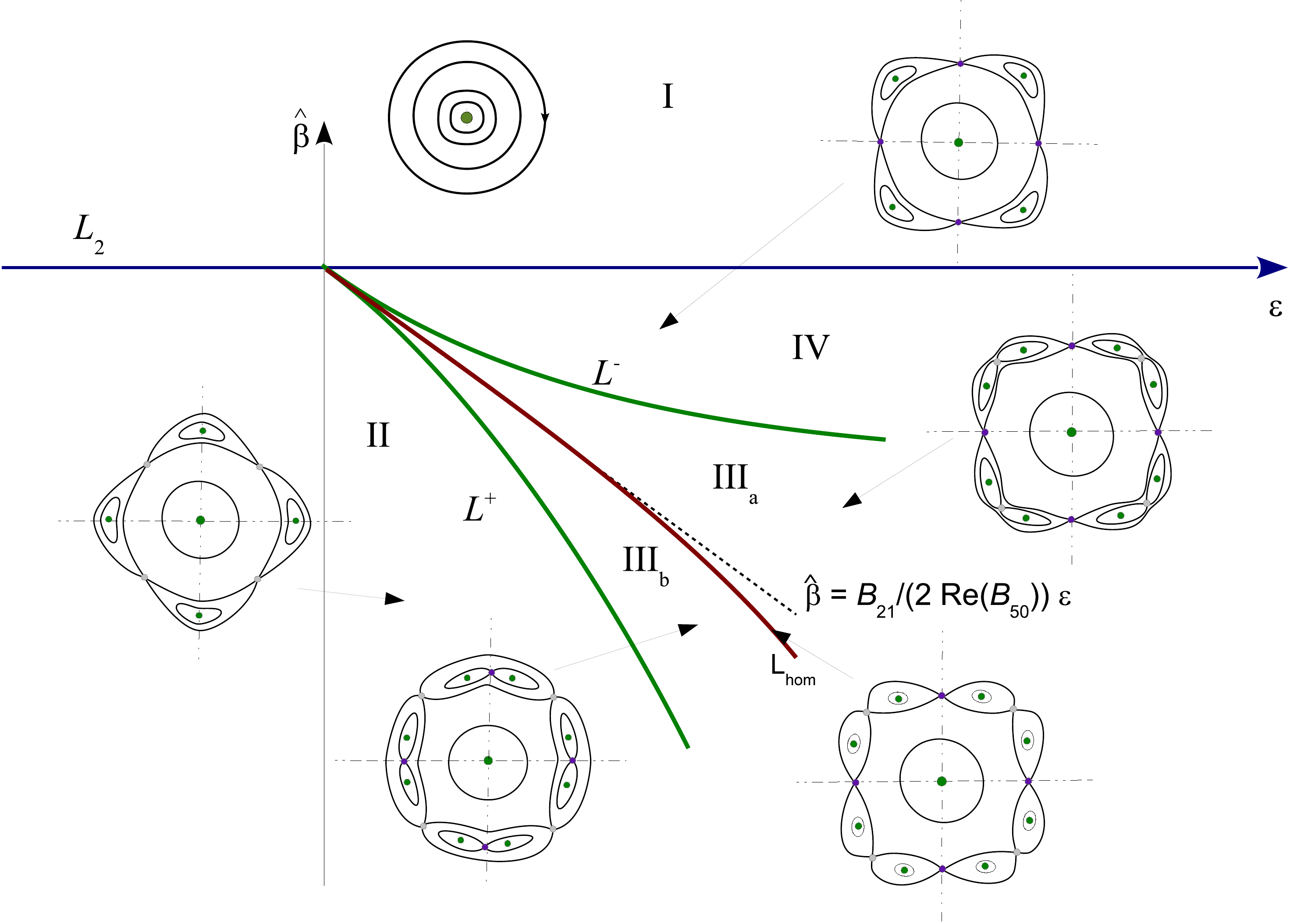}
\caption{ Main elements of the bifurcation diagram in case $B_{03}=0$ (we assume $\text{Re}(B_{50})<0$ and $B_{21}>0$).
 }
\label{fig_pi2-B03}
\end{figure}

\section{The 1:4 resonance in map~$\mathbf{C}_-$ }\label{sec:map_mns}

Consider map (\ref{cubH1m}) with parameters $(M_1,M_2)$ on the 1:4 resonance curve
\begin{equation}
L^-_{\pi/2}\;:\; M_1^2 = \frac{4}{27}M_2 (M_2 -3)^2,
\label{eq:res-}
\end{equation}
see Fig.~\ref{fg:1p4menys_corves}. For $M_1>0$ (resp. $M_1<0$) map
~$\mathbf{C}_-$ has $P_{\pi/2}^- = (-\sqrt{M_2/3}, -\sqrt{M_2/3})$ (resp.
$P_{\pi/2}^+ = (\sqrt{M_2/3}, \sqrt{M_2/3})$) as a fixed point with eigenvalues
$\pm i$. Note that $L_{\pi/2}^-$ has a self-intersection point
($M_1=0,M_2=3$) where the map has  two fixed points $(-1,-1)$ and $(1, 1)$ with
{eigenvalues}~$\pm i$ simultaneously.
 
The unfolding of normal form (\ref{HeComplNew13}) at the fixed point
$P_{\pi/2}^{\pm}$ (taking the suitable sign) has coefficients $8 B_{21}(0) = -3
+ 3 M_2,\; 8 B_{03} (0)= -1 - 3 M_2$, see~\cite{MGon05}.  Since $M_2\geq 0$ in
curve~$L_{\pi/2}^-$, we get $B_{03}(0) <0$.  Consequently, one has
$$\displaystyle A =  \frac{|3 - 3 M_2|}{1 + 3 M_2},$$ and all the three cases
$A>1$, $A<1$ (see Fig.~\ref{nondegcub}), and $A=1$ (see Fig.~\ref{fig_pi2-A1}) take place for $\mathbf{C}_-$.  The case $A=1$ occurs at
the points $P^+(M_1 = 16/27, M_2 = {1}/{3})$  and $P^-(M_1 = -16/27, M_2 =
{1}/{3})$, both in curve $L^-_{\pi/2}$, see Fig.~\ref{fg:1p4menys_corves}.

The bifurcation diagram of the 1:4 resonance in map~$\mathbf{C}_-$ is
displayed in Fig.~\ref{fg:1p4menys_corves}, where the bifurcation
curves~$L_{\pi/2}^-$, $L_4^i$ and~$\tilde L_4^i$, $i=1, 2, 3, 4$,  are shown in
the~$(M_1,M_2)$-parameter plane. The equation for~$L_{\pi/2}^-$ is written
in~(\ref{eq:res-}).  The curves $L_4^i$, $i=1,\dots, 4$, are the curves of
parabolic 4-periodic orbits with double eigenvalue 1.  The curves $\tilde
L_4^i$, $i=1, \dots, 4,$ are the curves of parabolic 4-periodic orbits with
double eigenvalue -1. Concretely:

\begin{itemize}
\item The curves $L^{1,2}_4$, whose equations are
\begin{equation} \label{eqL412}
L_4^{1,2} :\;    27  M_1^2  = 4 (1 + M_2)^{3},\;  \text{ where } M_2 > 1/3 \text{ and } M_1 > 0 \text{ for } L^1_4 \text{ while } M_1 < 0 \text{ for } L^2_4,
\end{equation}
are quadratically tangent to the curve~$L_{\pi/2}^-$
at the points~$P^\pm$ (case $A=1$). When crossing $L^{1,2}_4$ from bottom to
top two 4-periodic orbits are created as a result of a parabolic
(elliptic-hyperbolic) bifurcation.  Note that curves $L^{1,2}_4$ come from
the $L_3$ curve in Fig.~\ref{fig_pi2-A1}.  

\item The curve $L_4^3$, given by
\begin{equation} \label{eqL43}
 L_4^3 :\;   27 M_1^2 = 4 (2 + M_2)^2 (M_2 - 1),
\end{equation}
is associated with pitchfork bifurcations of elliptic 4-periodic orbits. This
is a consequence of the fact that the periodic orbit has a point on the
symmetry line $y=x$, see details in the proof of Lemma~\ref{lm:4par_minus}. 

\item The curve $L_4^4$, with equation
\begin{equation} \label{eqL44}
L_4^4 : \; 27 M_1^2  = 4  (M_2-2)^{3},
\end{equation}
corresponds to a parabolic bifurcation curve for 4-periodic (non-symmetric)
orbits.  These orbits are not of Birkhoff type since they are not ordered orbits
surrounding the elliptic fixed point $P^{\pm}_{\pi/2}$.

This curve has a special property. For parameters on $L_4^4$ with $M_1>0$
(resp. $M_1<0$) the fixed point $Q_-=(- \sqrt{(M_2-2)/3}, - \sqrt{(M_2-2)/3})$
(resp. $Q_+= -Q_-$) of map $\mathbf{C}_-$ is created at a parabolic
bifurcation. That is, both parabolic bifurcations, for the 4-periodic orbit and the fixed
point, take place simultaneously. See details below. This peculiarity is due to
the simple form of the cubic H\'enon map and clearly is not a persistent
property under arbitrary small perturbations.

\item The curves $\tilde L_4^i$, $i=1, \dots, 4,$ are the curves of parabolic 4-periodic
orbits with double eigenvalue -1. We have computed these curves
numerically. The curves $\tilde L_4^i$, $i=1, 2, 3$ correspond to period-doubling
bifurcations of elliptic 4-periodic orbits. 
\end{itemize}

\begin{figure}[tb]
\begin{center}
%\psfrag{ Lmpi2}{$L_{\pi/2}^-$}
%\psfrag{ L14}{$L_4^1$}
%\psfrag{ L24}{$L_4^2$}
%\psfrag{ L34}{$L_4^3$}
%\psfrag{ L44}{$L_4^4$}
%\psfrag{ Lt14}{$\tilde L_4^1$}
%\psfrag{ Lt24}{$\tilde L_4^2$}
%\psfrag{ Lt34}{$\tilde L_4^3$}
%\psfrag{ Lt44}{$\tilde L_4^4$}
%\psfrag{ Pm}{$P^-$}
%\psfrag{ PM}{$P^+$}
%\includegraphics[width=0.48\textwidth]{1p4menys-1.eps} 
%\includegraphics[width=0.48\textwidth]{1p4menys_d.eps}

\hspace{0.3cm}
\includegraphics[]{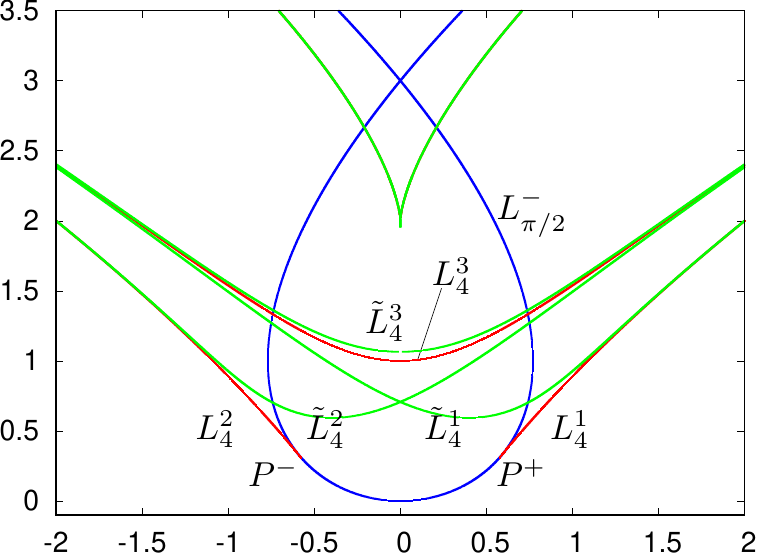} 
\hspace{1cm}
\includegraphics[]{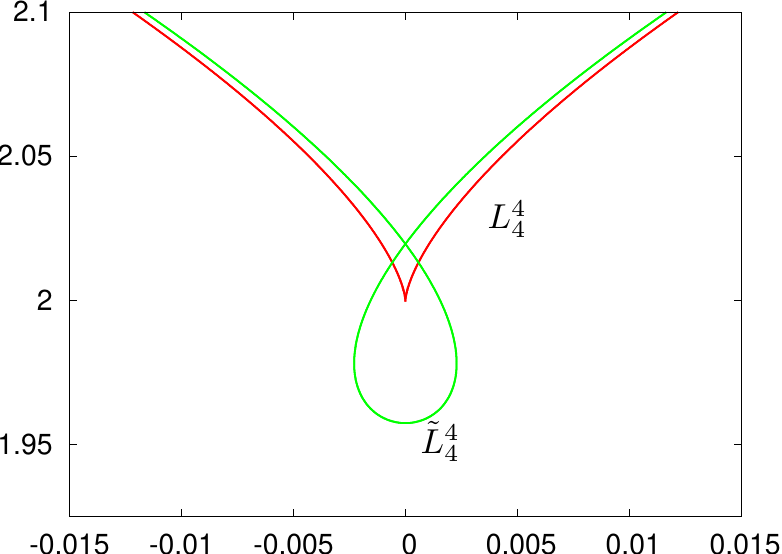}

\caption{Left: Bifurcation curves for the map~$\mathbf{C}_-$. Right: A zoom of the domain near $M_1=0, M_2=2$. }
\label{fg:1p4menys_corves}
\end{center}
\end{figure}

\begin{remark}
The curves $L^{1,2}_4$ together with $L_{\pi/2}^-$ reflect a peculiarity of
local aspects of the 1:4 resonance in the case of cubic map $\mathbf{C}_-$.
On the other hand, the curves $L_4^3$ and $L_4^4$ and the related bifurcations 
can be considered as peculiarities of the global aspects. 
Moreover, the curve $L_4^4$ has no direct relation with the problem of 1:4
resonance since it is a parabolic bifurcation curve for simultaneously two
non-symmetric 4-periodic orbits which are symmetric one to other with respect
to involution $R:(x,y) \rightarrow (y,x)$.  
\end{remark}

Let us give further details on how bifurcation curves organize the parameter
space. Fix a vertical line $M_1=C$ with $|C| > 4/3\sqrt{3}$.  For parameters
on this line, as we change $M_2$ from bottom to top, one has the following
sequence of bifurcations:
\begin{itemize}
\item For $C>4/3\sqrt{3}$ (resp. $C<-4/3\sqrt{3}$) the elliptic 4-periodic 
orbit created on $L_4^1$ (resp. $L_4^2$) becomes hyperbolic when crossing
$\tilde L_4^1$ (resp. $\tilde L_4^2$) and, at the crossing of this
period-doubling bifurcation curve, an elliptic 8-periodic orbit is born.
\item For larger $M_2$, there appears an elliptic 4-periodic orbit at the
inverse period-doubling bifurcation that takes place on $\tilde L_4^2$ (resp.
$\tilde L_4^1$). 
\item  This elliptic 4-periodic orbit undergoes a pitchfork bifurcation when
crossing $L_4^3$ and two elliptic 4-periodic orbits persist. The location of
the latter elliptic orbits is symmetric with respect to involution $R:
(x,y)\to (y,x)$. See more details in the proof of Lemma~\ref{lm:4par_minus}.

\item Those symmetric orbits undergo a period-doubling bifurcation at $\tilde L_4^3$.
\end{itemize}
In Fig.~\ref{M1_0p8} we show a sequence of phase space plots where these
bifurcations can be observed. The illustrations are for parameters on the
vertical line $M_1=0.8$.  The Fig.~\ref{M1_0p8} top left corresponds to
$M_2=0.65$, which is located between $L_4^1$ and $\tilde L_4^1$. We see the
main island around the fixed point and the 4-periodic satellite islands. There
are an elliptic 4-periodic orbit and a hyperbolic 4-periodic  one (this is
located near the peaks of the main stability island, the invariant
manifolds of each one of this saddle 4-periodic points surround the
corresponding satellite island). In Fig.~\ref{M1_0p8} top center we see a
magnification of the 4-periodic satellite island of stability around the
elliptic 4-periodic  orbit. This elliptic orbit undergoes a period-doubling
bifurcation when crossing $\tilde L_4^1$, the two stability islands can be seen in 
Fig.~\ref{M1_0p8} top right for $M_2=0.748$. Finally the Fig.~\ref{M1_0p8}
bottom left and right, for $M_2=1.36$ and $M_2=1.38$ respectively, illustrate
the pitchfork bifurcation taking place when crossing $L_4^3$.	

\begin{figure}[tb]
\begin{center}
\includegraphics[width=0.35\textwidth]{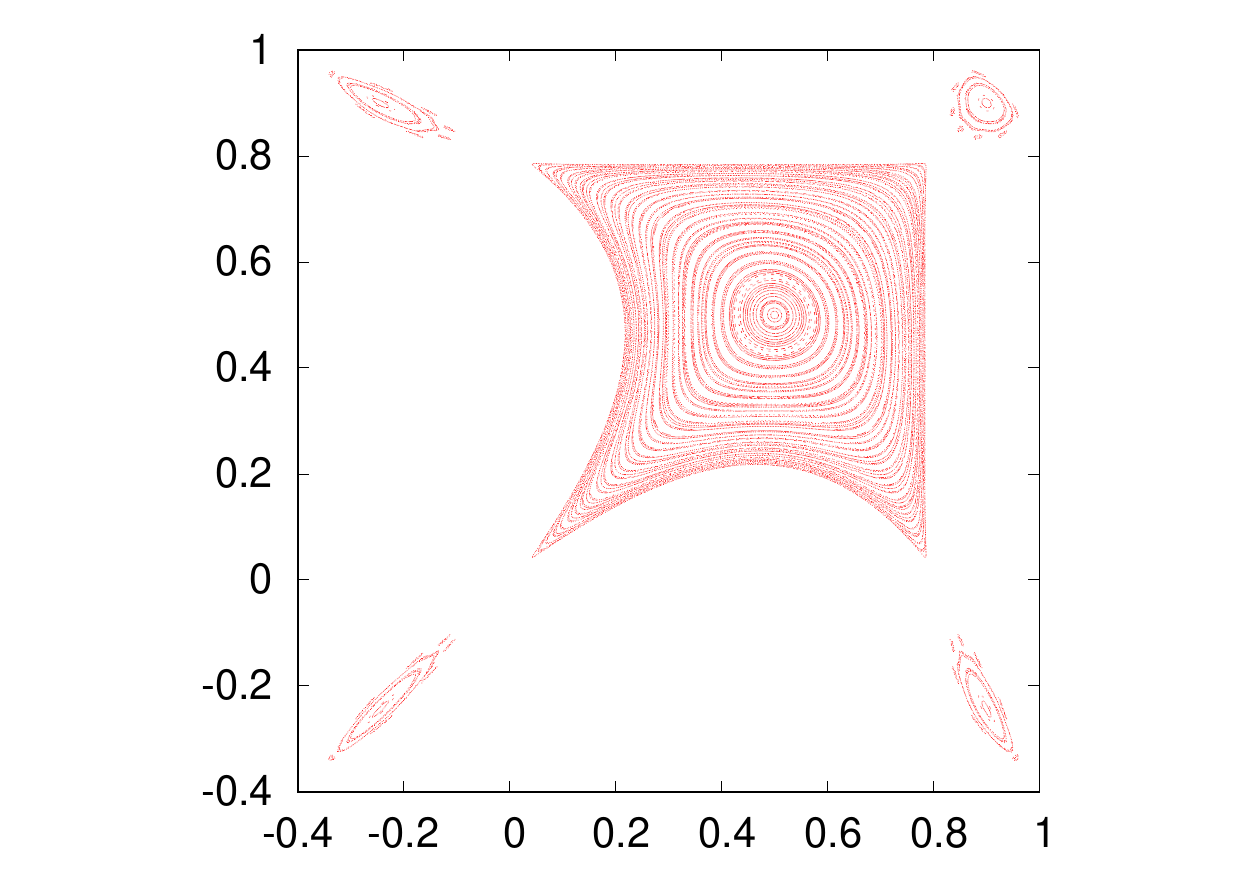} 
\hspace{-1cm}
\includegraphics[width=0.35\textwidth]{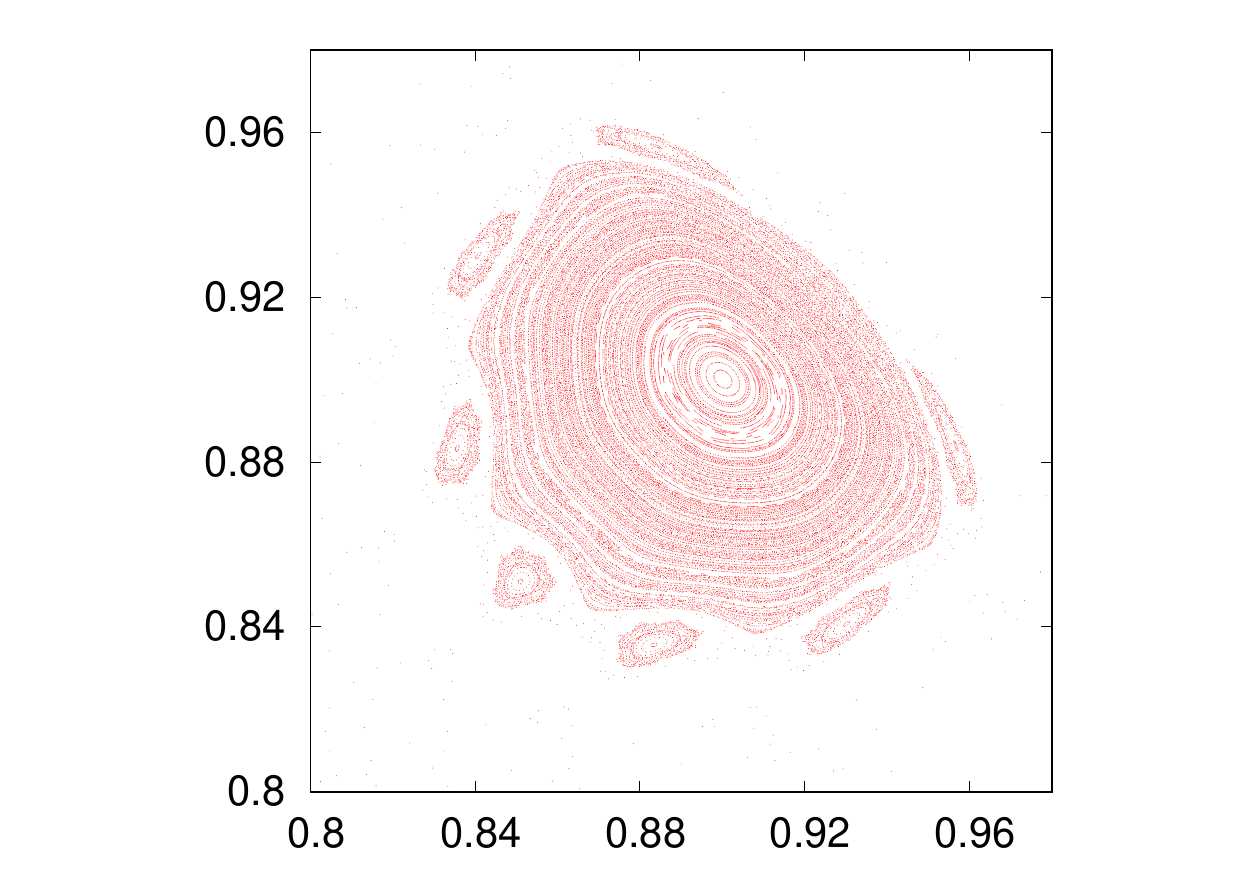} 
\hspace{-1cm}
\includegraphics[width=0.35\textwidth]{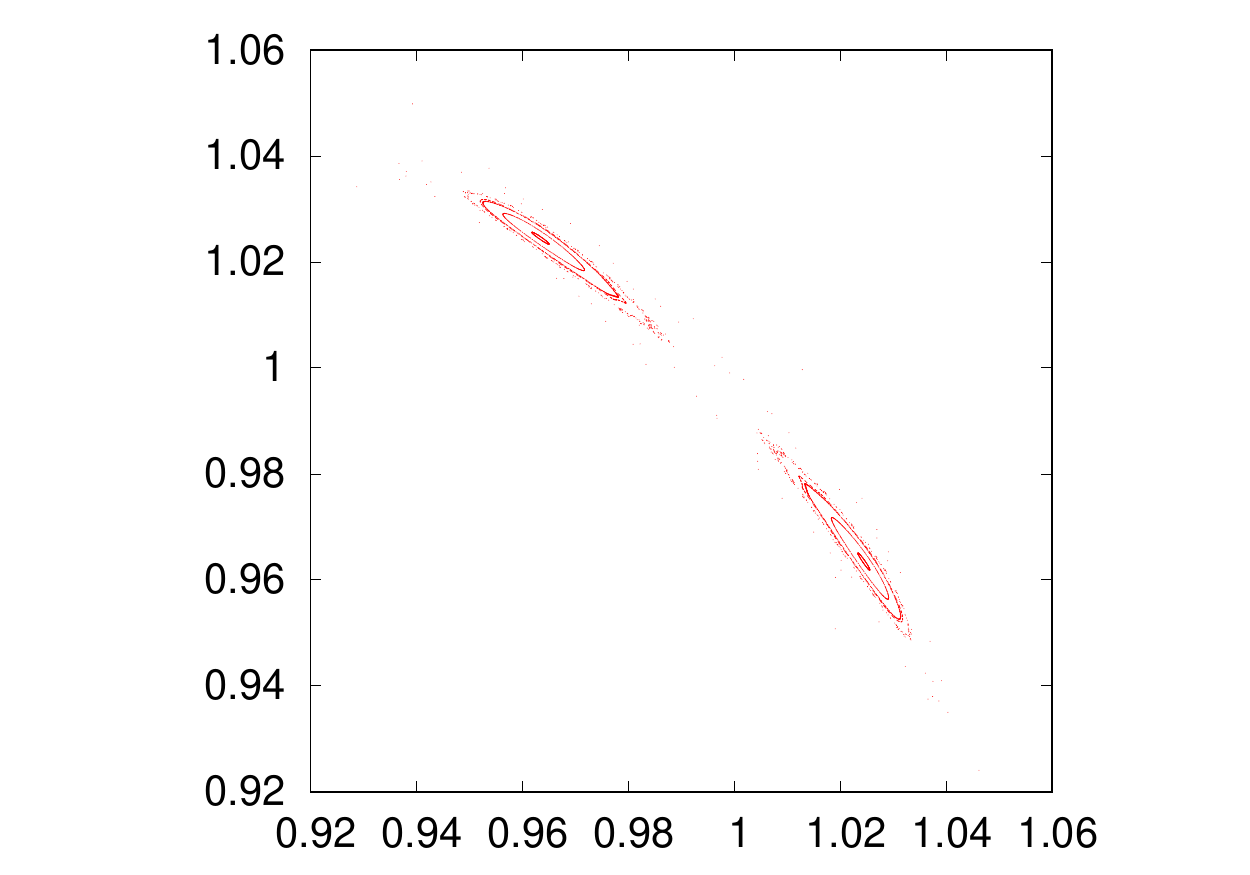}  \\
\includegraphics[width=0.35\textwidth]{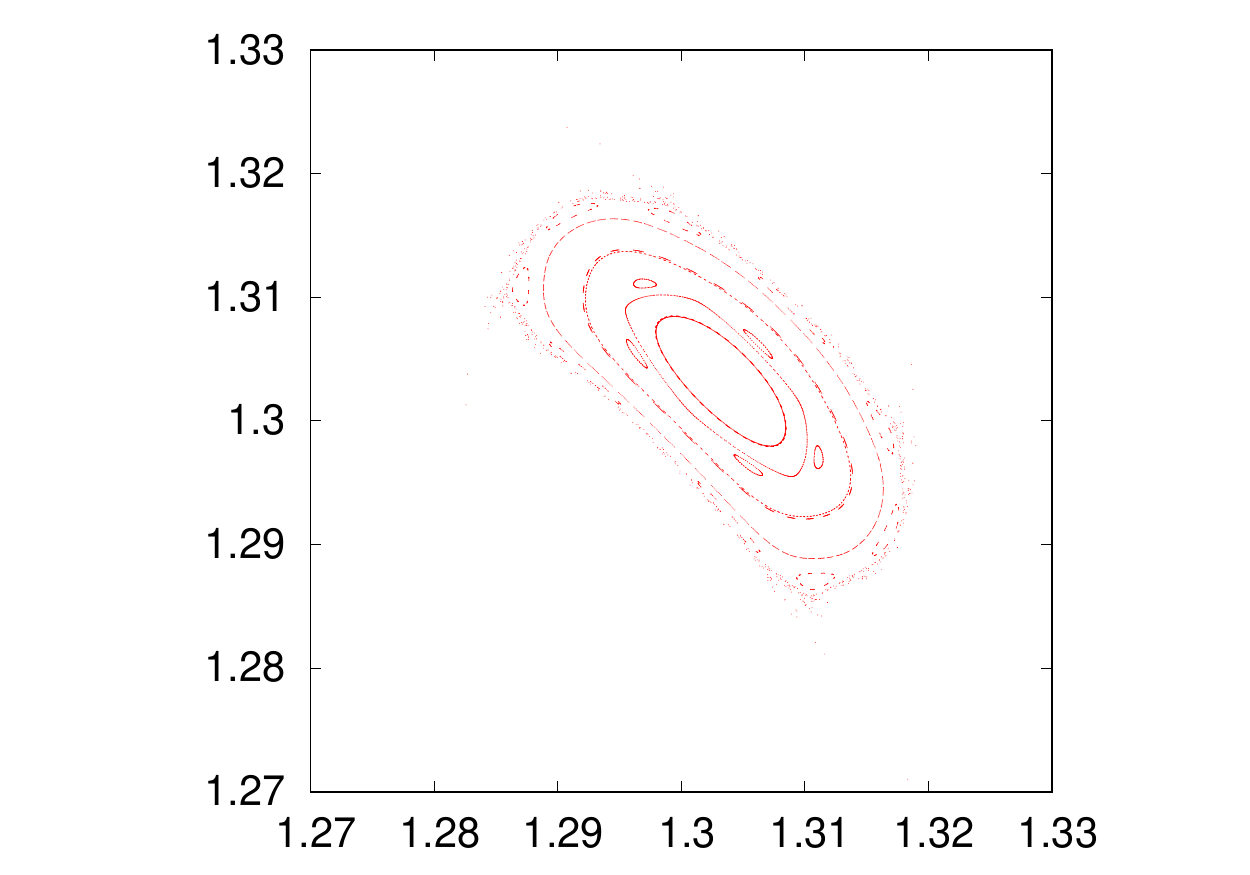} 
\includegraphics[width=0.35\textwidth]{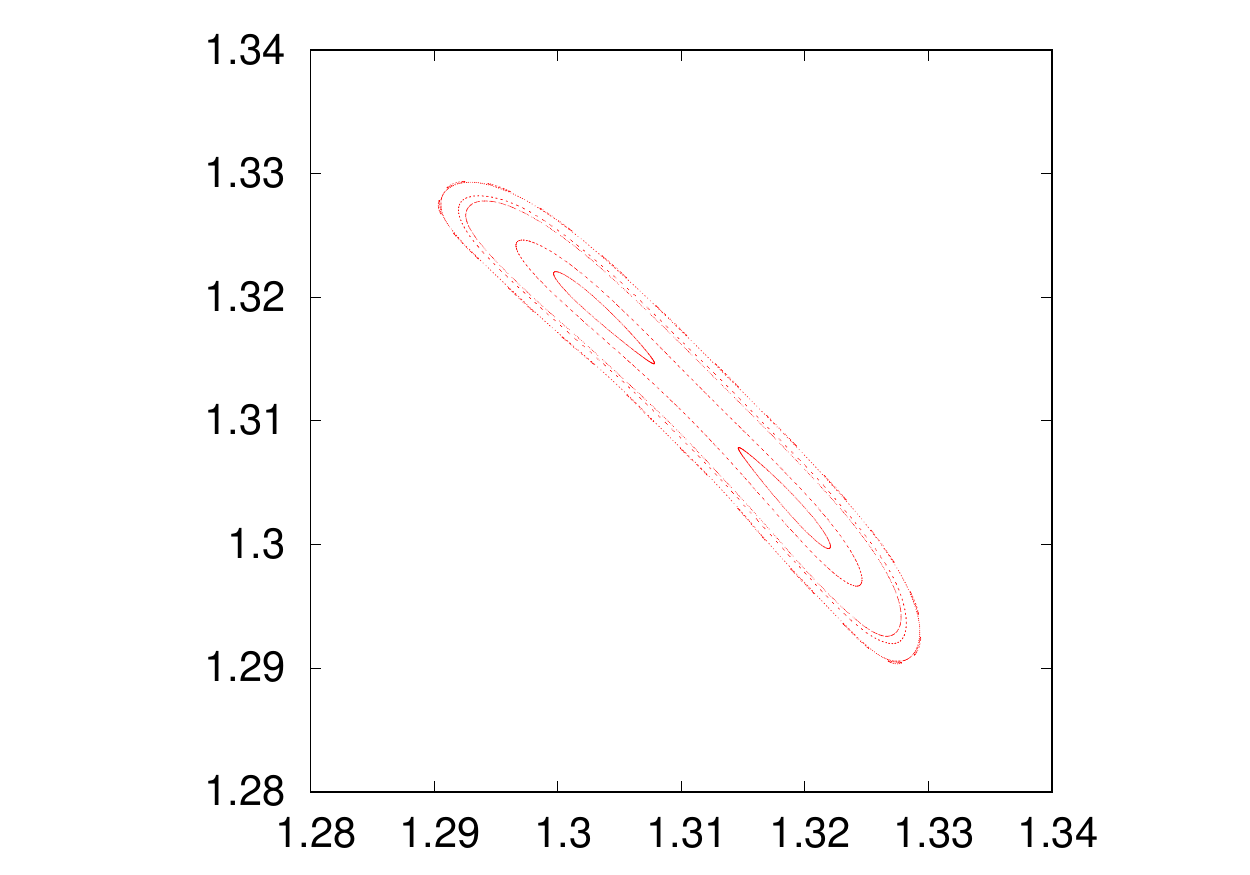} 
\caption{Sequence of bifurcations taking place on the line $M_1=0.8$. In the
top left plot, for $M_2=0.65$,  we see the 4-periodic satellite island
surrounding the main island around the fixed point.  The top center plot is a
magnification of the 4-periodic satellite island seen in the top left figure.
When increasing  $M_2$ the related elliptic 4-periodic  orbit bifurcates.  The
top right plot is for $M_2=0.748$, after the crossing of $\tilde L_4^1$. The
bottom left plot, for $M_2=1.36$, and the bottom right, for $M_2=1.38$, show
the island before and after the pitchfork bifurcation curve $L_4^3$.
  }
\label{M1_0p8}
\end{center}
\end{figure}

Curve $L_4^4$ is the curve of parabolic 4-periodic orbits with double
eigenvalue 1 which is a parabolic bifurcation curve for the
4-periodic orbits. As already said, this curve also coincides
with the curve which corresponds to a parabolic bifurcation of
the fixed point (or, when $M_1=0, M_2=2$, to a pitchfork
bifurcation of the fixed point). Let us give further details on the sequence
of bifurcations when crossing the lines $\tilde L_4^4$ and $L_4^4$ in
Fig.~\ref{fg:1p4menys_corves} right. To this end consider the vertical line
$M_1=0.0003$. When moving $M_2$ from bottom to top in
Fig.~\ref{fg:1p4menys_corves} right one has the following sequence of bifurcations:
\begin{itemize}
\item First we have the crossing of $\tilde L_4^4$. At this crossing two
non-symmetric elliptic 4-periodic orbits and two non-symmetric hyperbolic
4-periodic orbits are created as a result of an inverse period-doubling
bifurcation. The position of the two elliptic 4-periodic orbits is shown in
Fig.~\ref{0p0003} top left. A magnification of one of the satellite islands is
shown in Fig.~\ref{0p0003} top center.
\item When crossing $L_4^4$ there is a parabolic bifurcation and two
new 4-periodic orbits, one elliptic and the other of saddle type, are created.
This can be seen in Fig.~\ref{0p0003} top right.
\item Increasing $M_2$, we have two consecutive crossings of $\tilde L_4^4$.
These correspond to period-doubling bifurcations of each of the two elliptic 4-periodic
 orbits. The two satellite islands after the period-doubling are shown
in Fig.~\ref{0p0003} bottom left and right, respectively.
\end{itemize}

\begin{figure}[h]
\begin{center}
\includegraphics[width=0.35\textwidth]{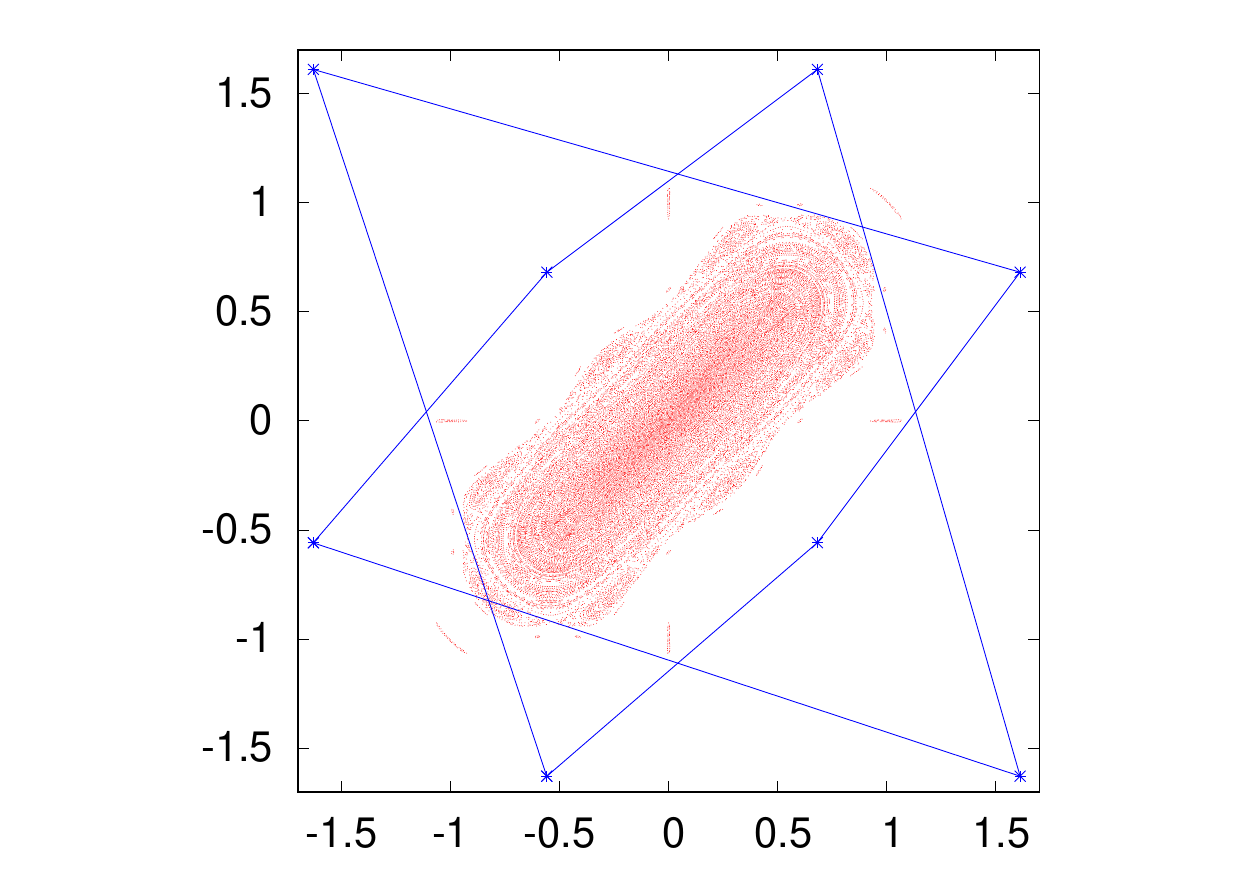} 
\hspace{-1cm}
\includegraphics[width=0.35\textwidth]{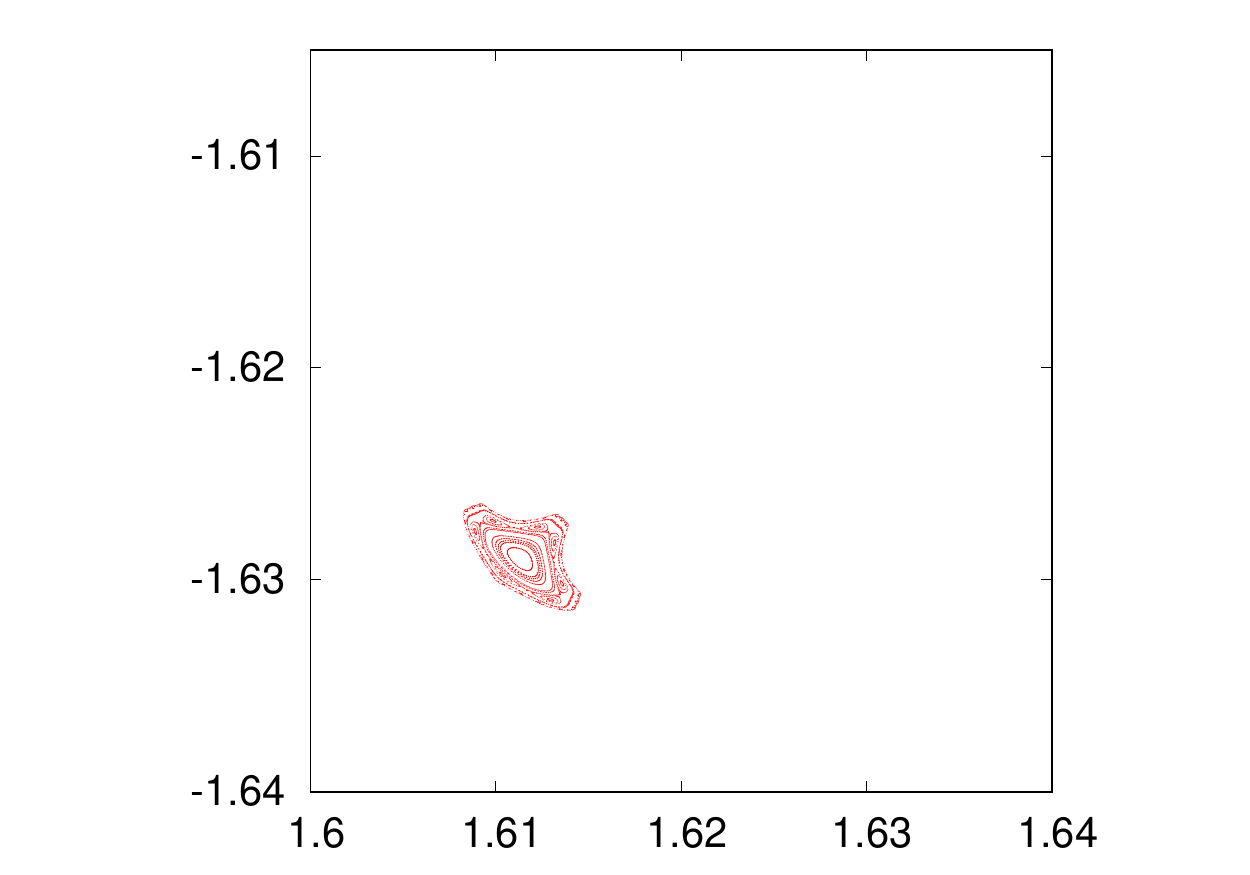} 
\hspace{-1cm}
\includegraphics[width=0.35\textwidth]{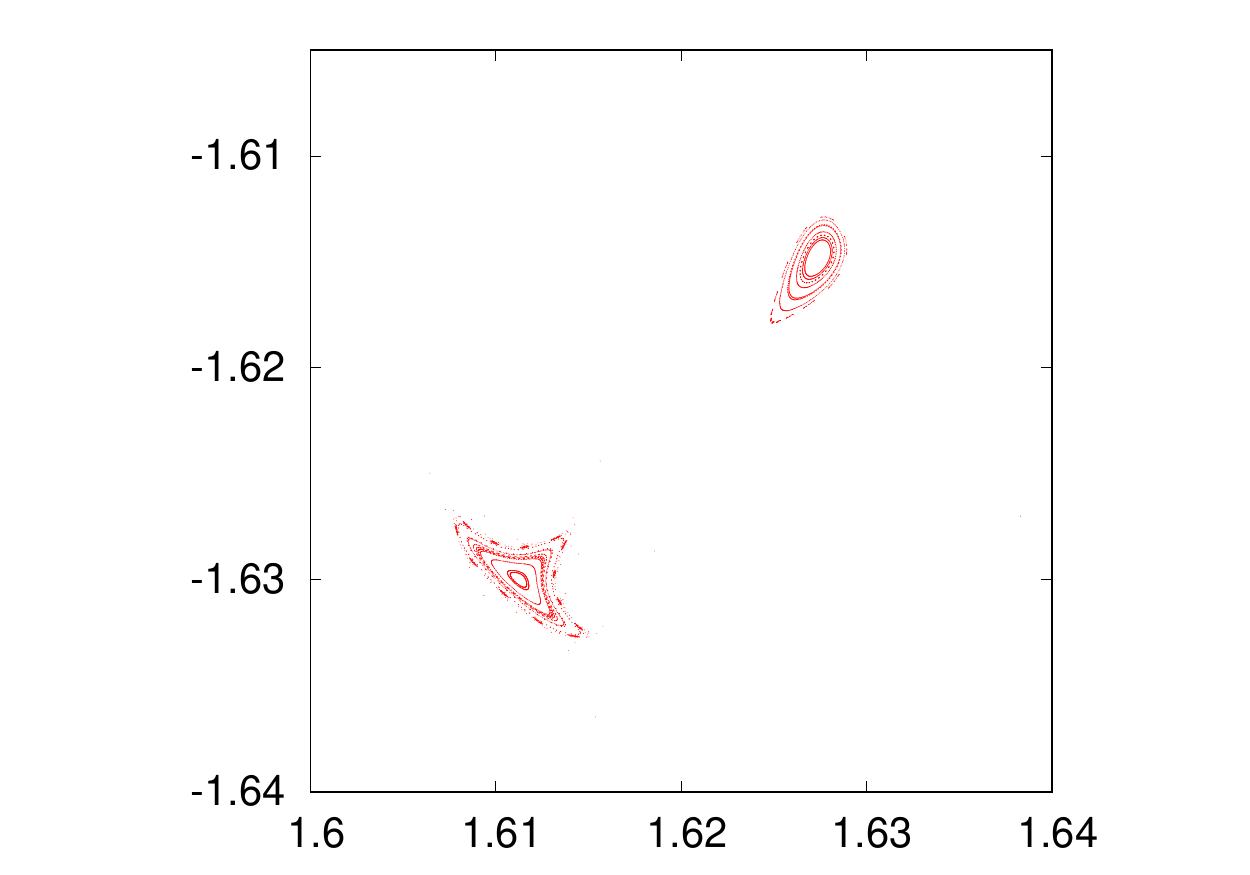}\\
\includegraphics[width=0.35\textwidth]{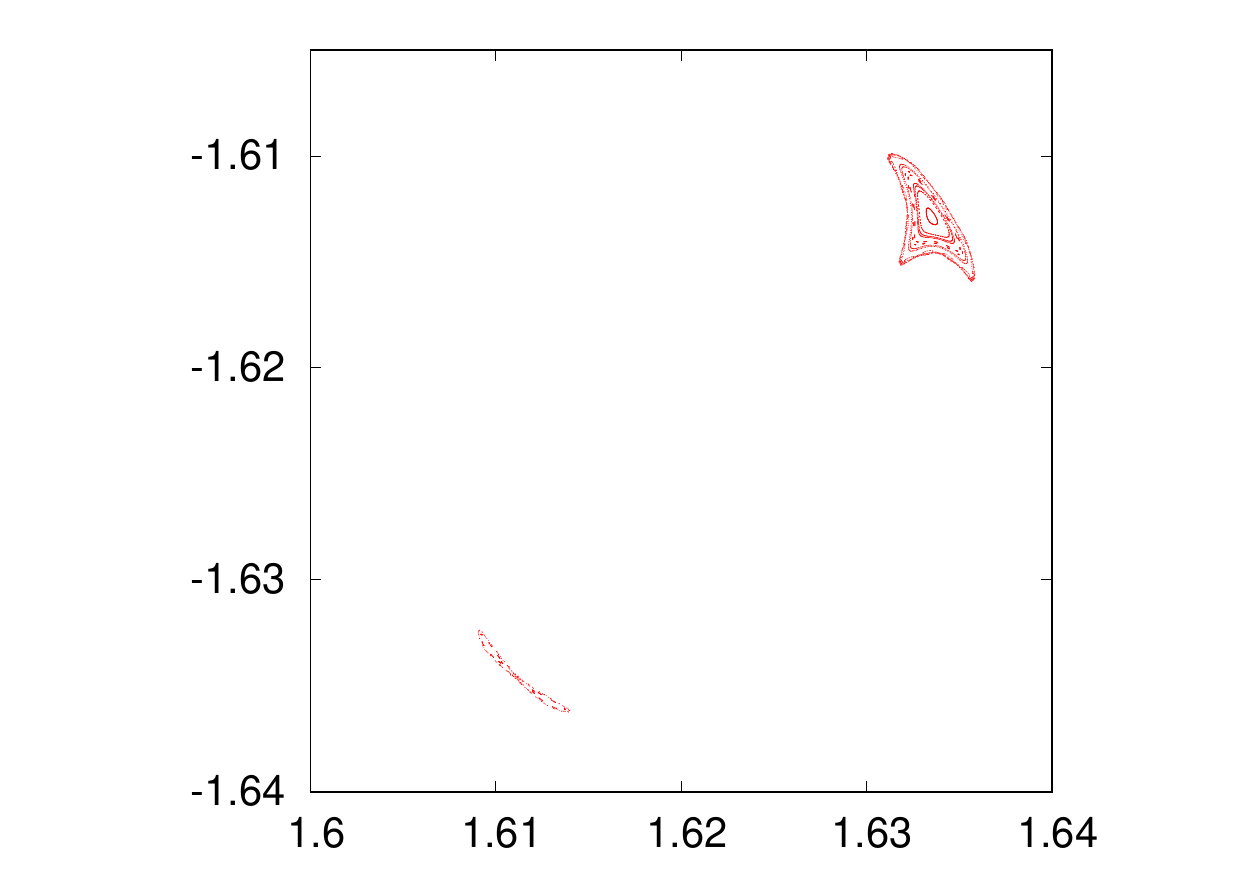} 
\includegraphics[width=0.35\textwidth]{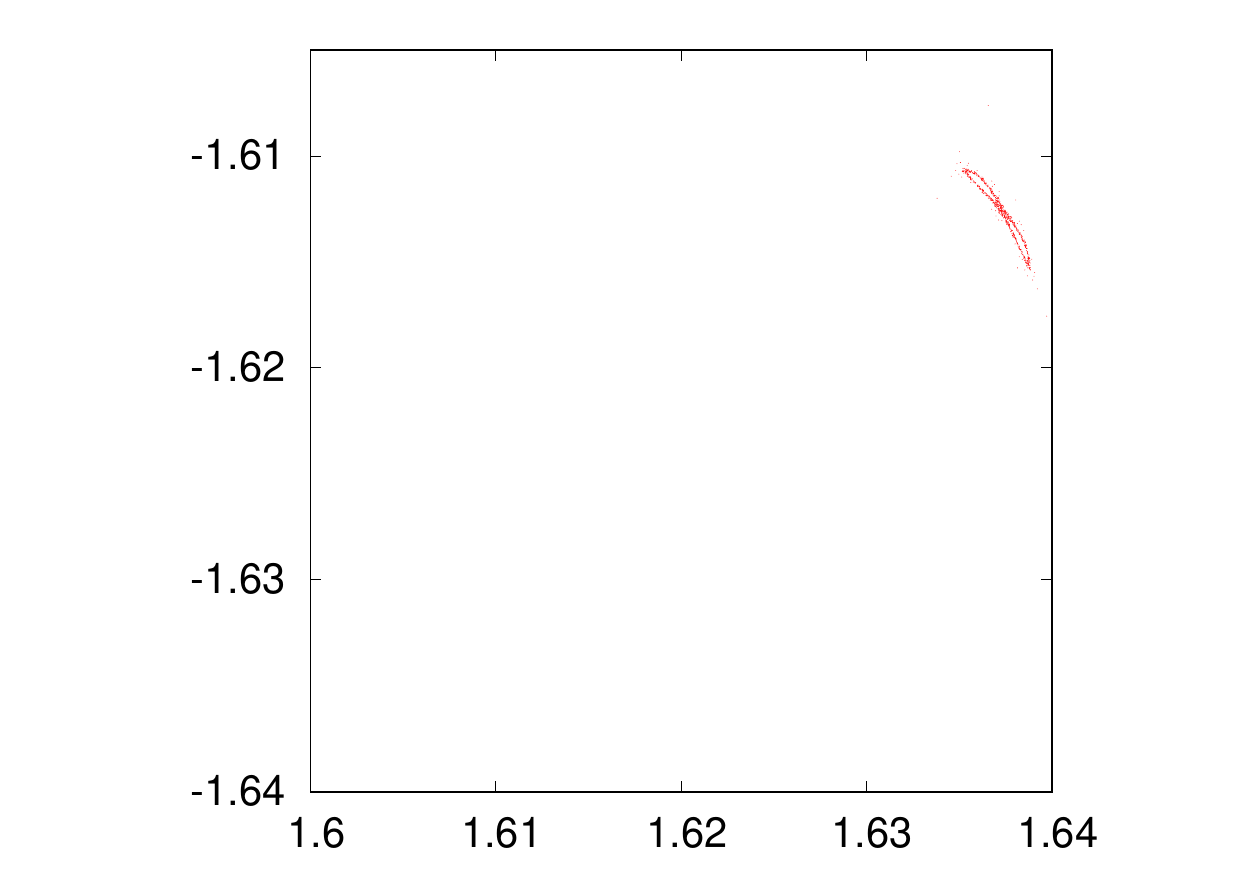} 
\caption{Sequence of bifurcations taking place on the line $M_1=0.0003$. Top
left: For $M_2=2.008$ we see the main island around the fixed point. The lines
connect the iterates of the two elliptic 4-periodic  orbits. Top center: a
magnification of the left plot showing one of the 4-periodic satellite islands.
Top right: For $M_2=2.0095$ we see the stability islands around one of the
iterates of the two elliptic 4-periodic orbits. The new island (at the top
right of the plot) is created when crossing the parabolic curve
$\tilde L_4^4$ in Fig.~\ref{fg:1p4menys_corves}.  Bottom left: For
$M_2=2.0171$ we see that one of the elliptic
4-periodic orbits undergoes a period-doubling bifurcation. Bottom right: For
$M_2=2.0232$ the remaining elliptic 4-periodic island undergoes a period-doubling
bifurcation.
  }
\label{0p0003}
\end{center}
\end{figure}

\section{The 1:4 resonance in map~$\mathbf{C}_+$.}\label{sec:map_pl} 

Map~$\mathbf{C}_+$ has the fixed point $P^{-}_{\pi/2} = (-\sqrt{-M_2/3},
-\sqrt{-M_2/3})$ (resp. $P^+_{\pi/2}=-P^-_{\pi/2}$)  with eigenvalues $\pm i$ for parameters
$(M_1,M_2)$ with $M_1>0$ (resp. $M_1<0$) on the 1:4 resonance curve
$$
L^+_{\pi/2}\;:\; M_1^2 = -\frac{4}{27}M_2 (M_2 -3)^2.
$$

The coefficients of normal form (\ref{HeComplNew13}) around $P^{\pm}_{\pi/2}$
 are $\displaystyle 8B_{21}(0) = 3 - 3 M_2, 8B_{03}(0)= 1 + 3 M_2$ \cite{MGon05}, hence 
$$
\displaystyle A   
= \frac{|3 - 3 M_2|}{|1 + 3 M_2|} = \left\{ \begin{array}{ll}
\displaystyle 1+ \frac{2-6M_2}{1+3M_2}, & \displaystyle -\frac{1}{3} < M_2 \leq 0, 
\vspace{0.2cm} \\
\displaystyle 1+ \frac{4}{|1+3 M_2|}, & \displaystyle M_2<-\frac{1}{3}.
\end{array}\right.
$$
Since $M_2\leq 0$ in~$L_{\pi/2}^+$, we always have $A>1$.  However, we get 
degeneracy $B_{03}(0) = 0$  at~$M_1=\pm 20/27$ and~$M_2 = -1/3$. These are the
points $P_4^l$ and $P_4^r$ in the bifurcation diagram shown in
Fig.~\ref{fig:res1_4p} left where we display in red the bifurcation curves
for which there are parabolic 4-periodic orbits with double eigenvalue 1 (either
parabolic or pitchfork bifurcations) and in green those curves for
which there are 4-periodic orbits with double eigenvalue -1 (hence
period-doubling bifurcations).

For parameters $(M_1,M_2)$ above the curve $L^+_{\pi/2}$ the local phase space
is topologically equivalent to that of region I in Fig.~\ref{fig_pi2-A1},
that is the fixed point $P^{\pm}_{\pi/2}$ is elliptic and there are no
4-periodic orbits surrounding it.

Crossing the curve $L^+_{\pi/2}$ through $(M_1,M_2) \in L^+_{\pi/2} \setminus
\{P_4^{r,l}\}$ is analogous to cross the line $L_2$ in Fig.~\ref{fig_pi2-A1}
from region I to II.  In particular, around $P^{\pm}_{\pi/2}$ there appear
a saddle 4-periodic orbit whose invariant manifolds bound a 4-periodic island
of stability with an  elliptic 4-periodic orbit inside.
 
\begin{figure}[tb]
%\psfrag{ Lmpi2}{$L_{\pi/2}^+$}
%\psfrag{ P4l}{$P_4^l$}
%\psfrag{ P4r}{$P_4^r$}
%\psfrag{ tC4r}{$\tilde C_4^r$}
%\psfrag{ tC4l}{$\tilde C_4^l$}
%\psfrag{ hB4l}{$\hat B_4^l$}
%\psfrag{ hB4r}{$\hat B_4^r$}
%\psfrag{ B4l}{$B_4^l$}
%\psfrag{ B4r}{$B_4^r$}
%\psfrag{ C4l}{$C_4^l$}
%\psfrag{ C4r}{$C_4^r$}

%\includegraphics[width=0.48\textwidth]{bif14_p.eps}
%\includegraphics[width=0.48\textwidth]{bif14_pD.eps}

\hspace{0.3cm}
\includegraphics[]{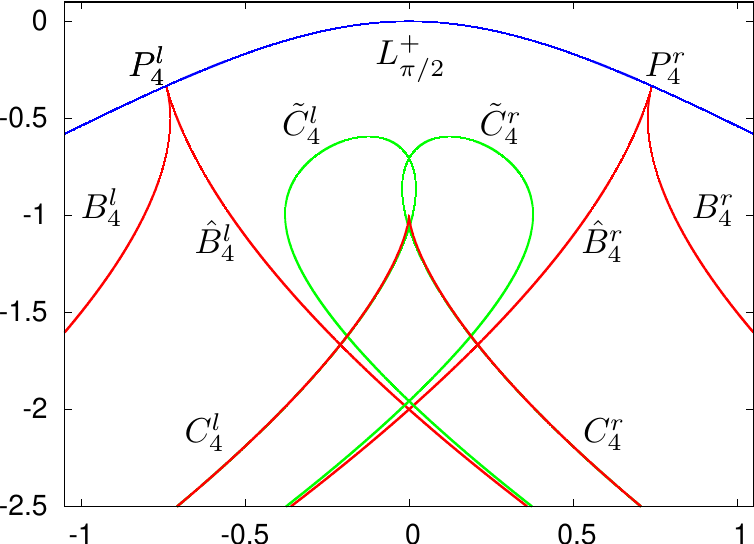}
\hspace{1cm}
\includegraphics[]{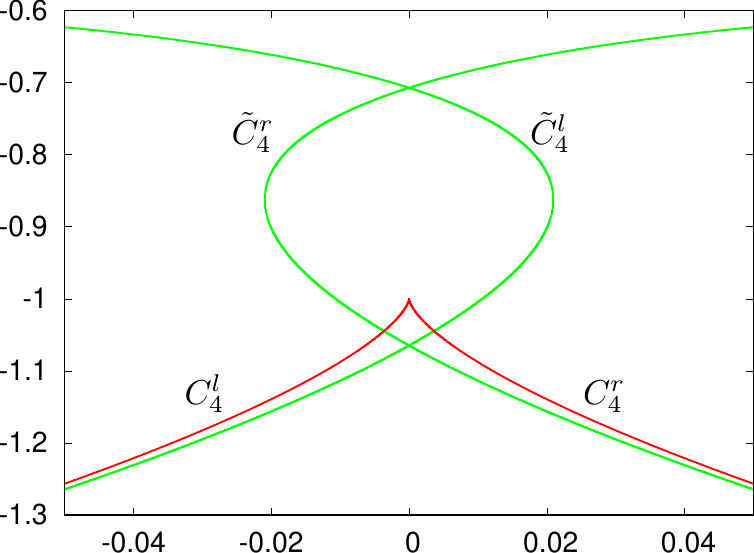}
\caption{Bifurcation curves in map $\mathbf{C}_+$. The right plot is a magnification of the left one
near $(0,-1)$.}
	\label{fig:res1_4p}
\end{figure}

The points $P_4^{l,r}$, corresponding to degeneracy $B_{03}=0$, are 
endpoints of the pitchfork bifurcation curves $B^l_4$, $\hat B^l_4$ and $B^r_4$,
$\hat B^r_4$. The explicit equations of such curves are
\begin{equation} \label{eqBlrhBlr}
\begin{array}{rrcl}
B^{l,r}_4:   & \displaystyle 27 M_1^2 & \! = \! & \left( 3 \sqrt{3} - 3 \sqrt{-M_2} - 2 M_2 \sqrt{-M_2}\right)^2, \;\; M_2 < -1/3. \vspace{0.2cm} \\
\hat B^{l,r}_4: & \displaystyle  27 M_1^2 & \! = \! &  4 (2 + M_2)^2 (1 - M_2), \quad M_2 < -1/3. \vspace{0.3cm} 
\end{array}
\end{equation}

In Fig.~\ref{fig:res14_p_M20p5} we display the sequence of bifurcations
taking place when getting inside/outside the region bounded by curves
$\hat B^r_4$ and $B^r_4$ (by symmetry, bifurcations through $B^l_4$ and
$\hat B^l_4$ are analogous).  For illustrations we choose $M_2=-0.5$
and change $M_1$. For $M_1=0.7$ (top left plot in
Fig.~\ref{fig:res14_p_M20p5}) we see the 4-periodic island having the
elliptic point on the symmetry line $y=x$. For parameters on $\hat B^r_4$ this
point undergoes a pitchfork bifurcation: the elliptic 4-periodic orbit becomes a
saddle 4-periodic orbit and a pair of elliptic 4-periodic orbits appear nearby,
see Fig.~\ref{fig:res14_p_M20p5} top center.  Then the separatrices of the
saddle 4-periodic orbits become larger, see  Fig.~\ref{fig:res14_p_M20p5} top
right, and at some moment between $M_1=0.718$ and $M_2=0.719$ the separatrices
merge\footnote{Note that for map $\mathbf{C}_+$ the invariant manifolds are
not expected to exactly merge due to the splitting of separatrices.} with
the exterior separatrices of the other saddle 4-periodic orbits and a
sequence of bifurcations related to the reconstruction of homo/heteroclinic connections
takes place, see  Fig.~\ref{fig:res14_p_M20p5} bottom left.  After that, an
inverse pitchfork bifurcation occurs, the two elliptic 4-periodic collide into
the saddle 4-periodic orbit which becomes elliptic, see
Fig.~\ref{fig:res14_p_M20p5} bottom center and right.  Note that for
$M_1=0.73$ the saddle 4-periodic orbit is on the symmetry line $y=x$. 
This sequence of bifurcations (pitchfork, reconnection of the invariant
manifolds and inverse pitchfork) happens generically at the unfolding of the
degenerate case $B_{03}=0$, see details in Appendix~\ref{appendix1p4deg}.

\begin{figure}[tb]
	\centering
        \hspace{-0.5cm}
	\includegraphics[height=5cm]{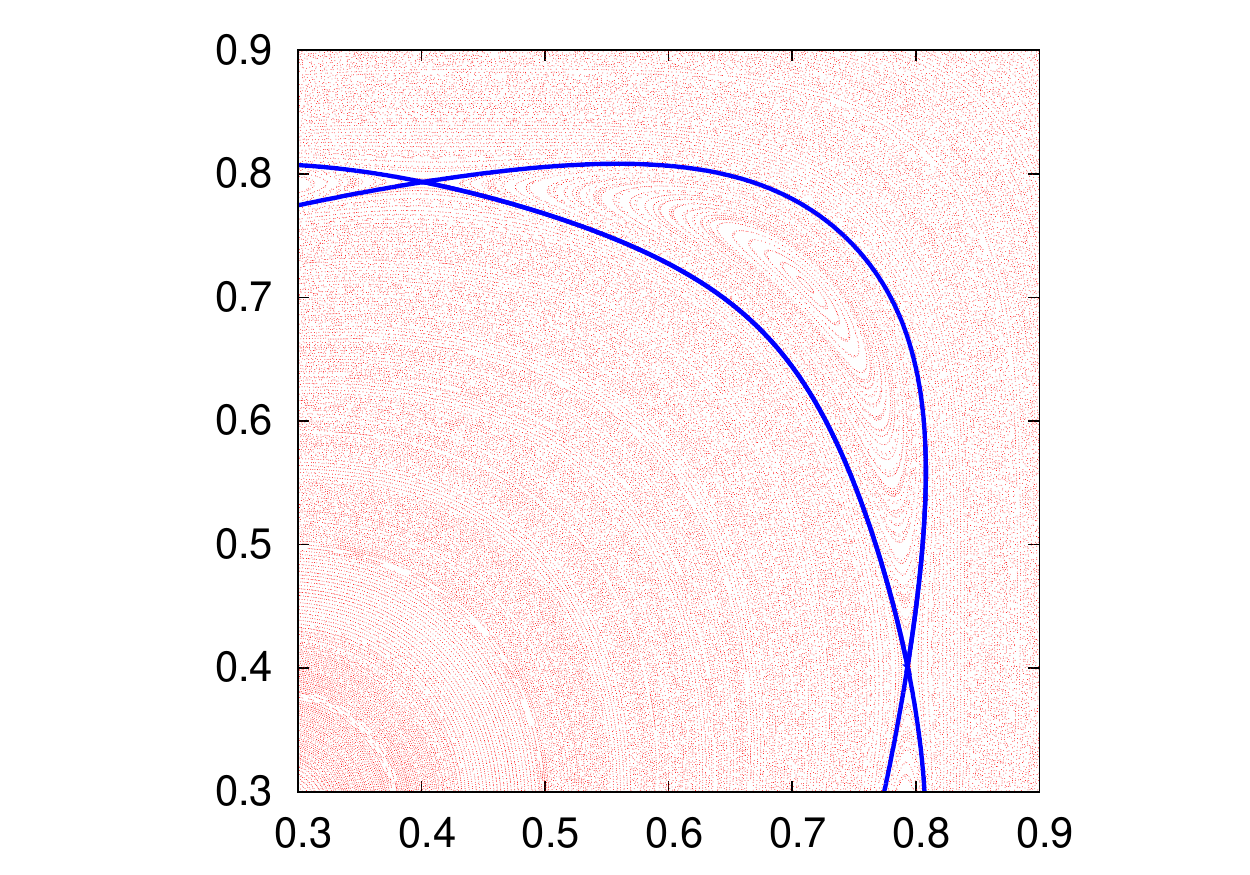}
        \hspace{-2.2cm}
	\includegraphics[height=5cm]{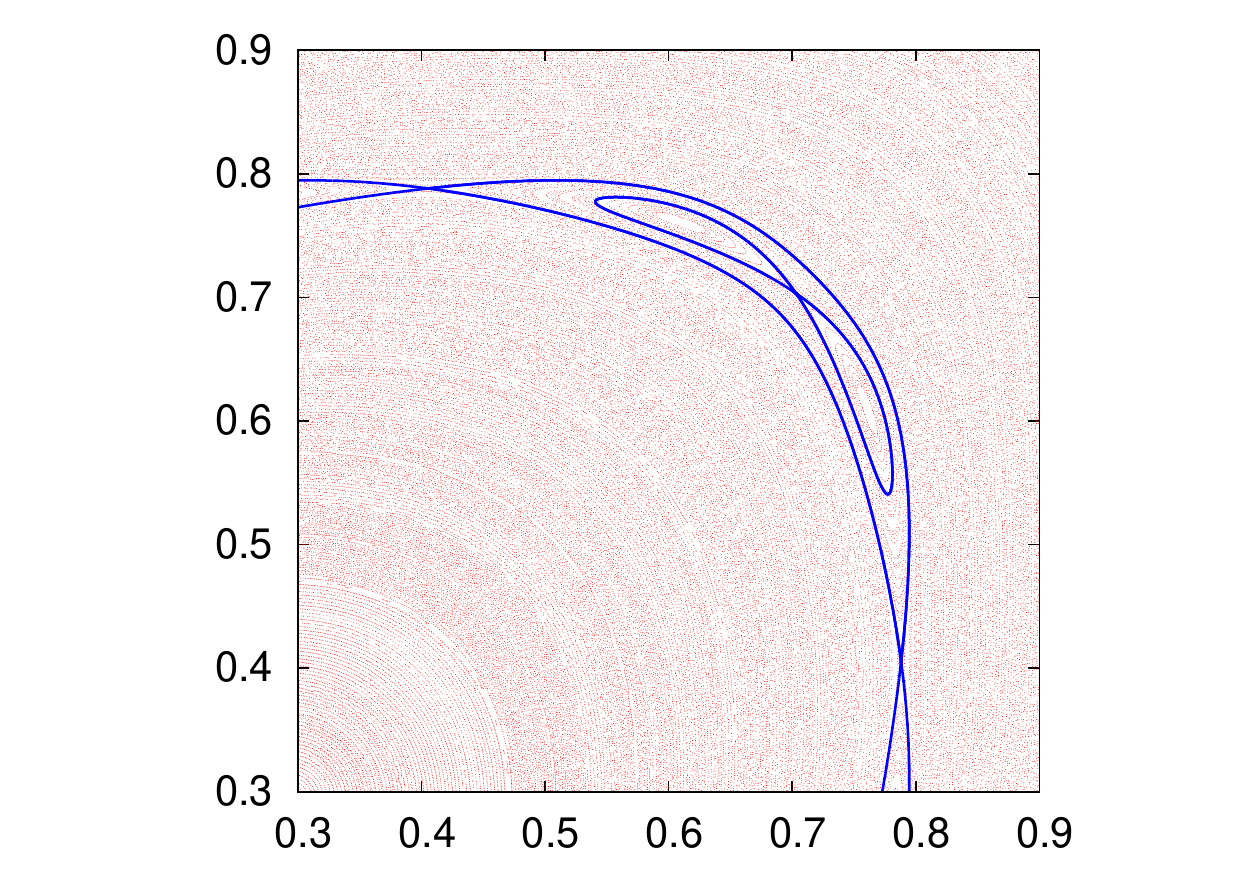}
        \hspace{-2.2cm}
	\includegraphics[height=5cm]{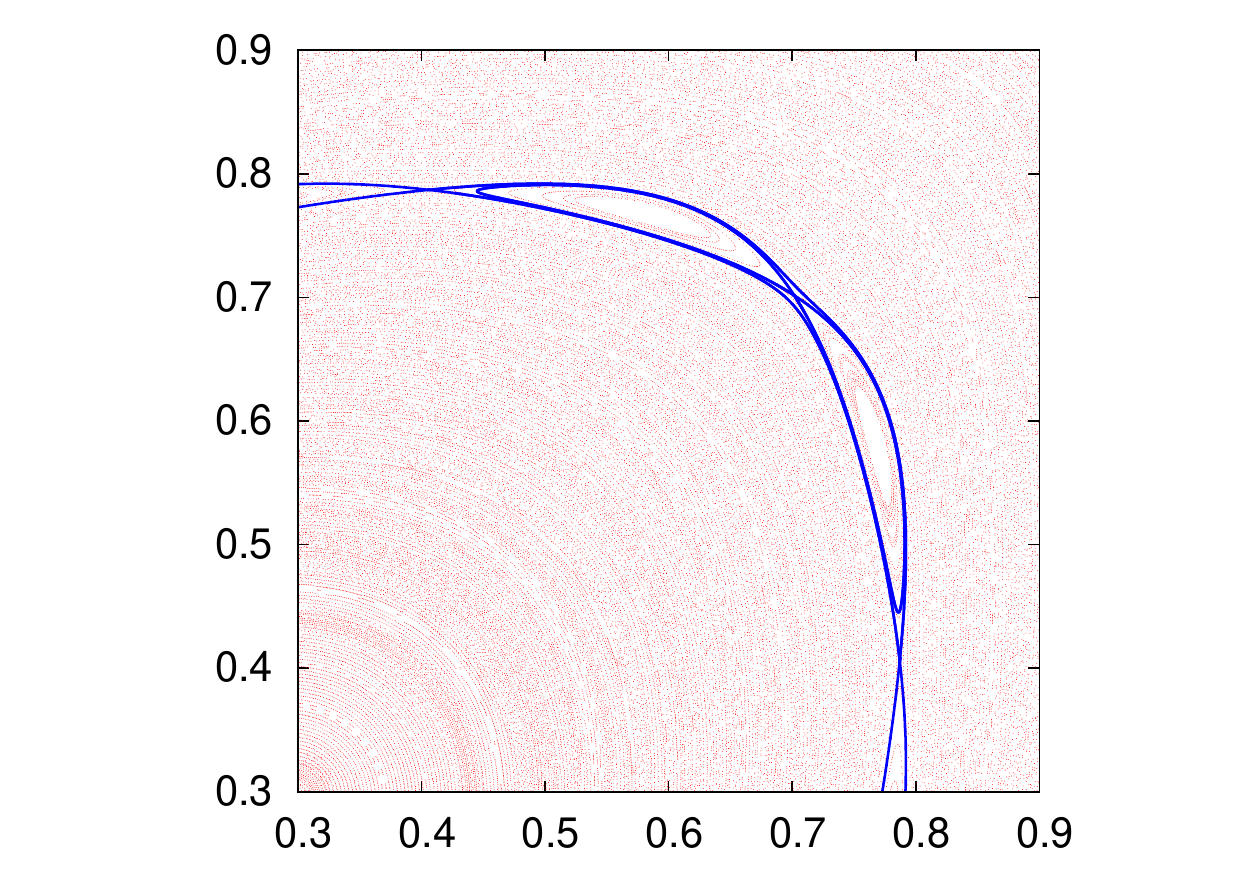} \\ 
        \hspace{-0.5cm}
	\includegraphics[height=5cm]{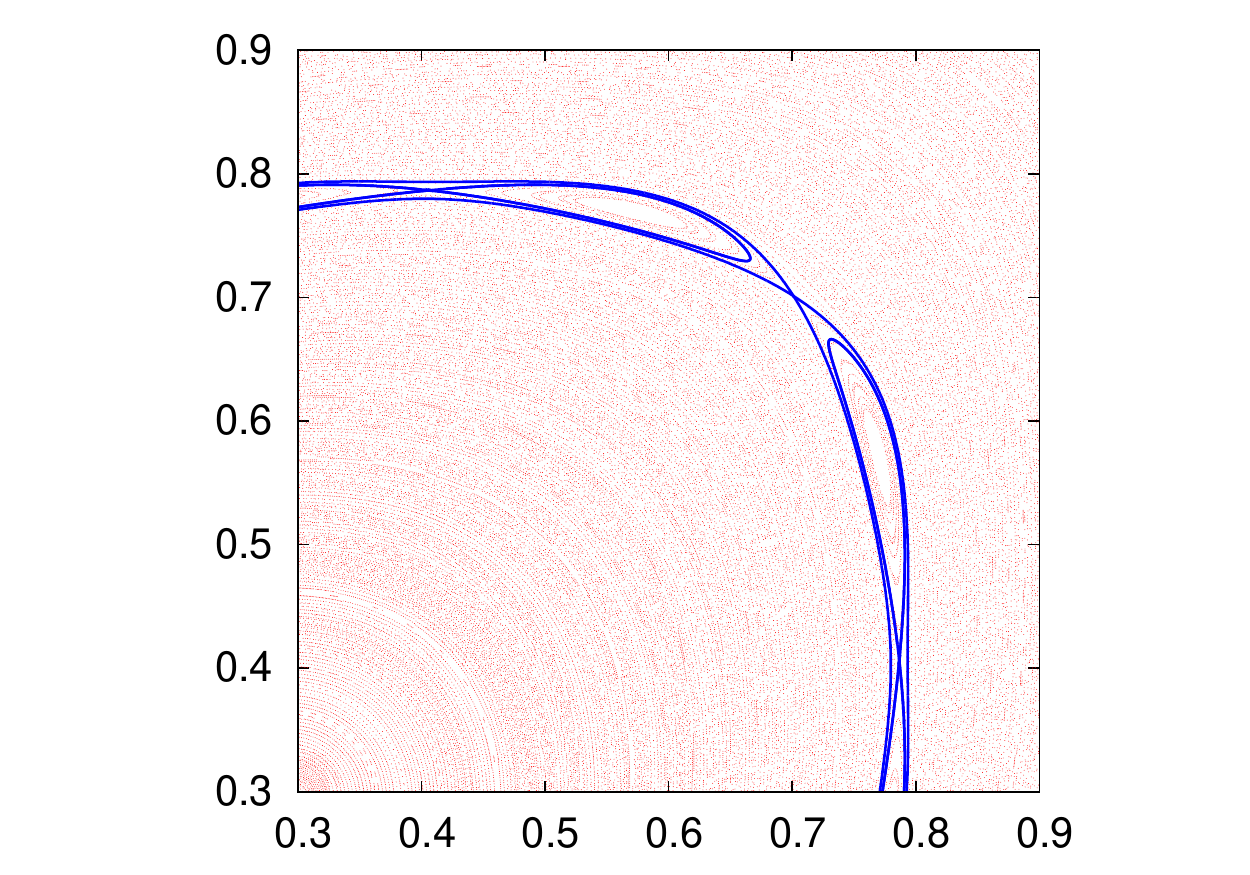} 
        \hspace{-2.2cm}
	\includegraphics[height=5cm]{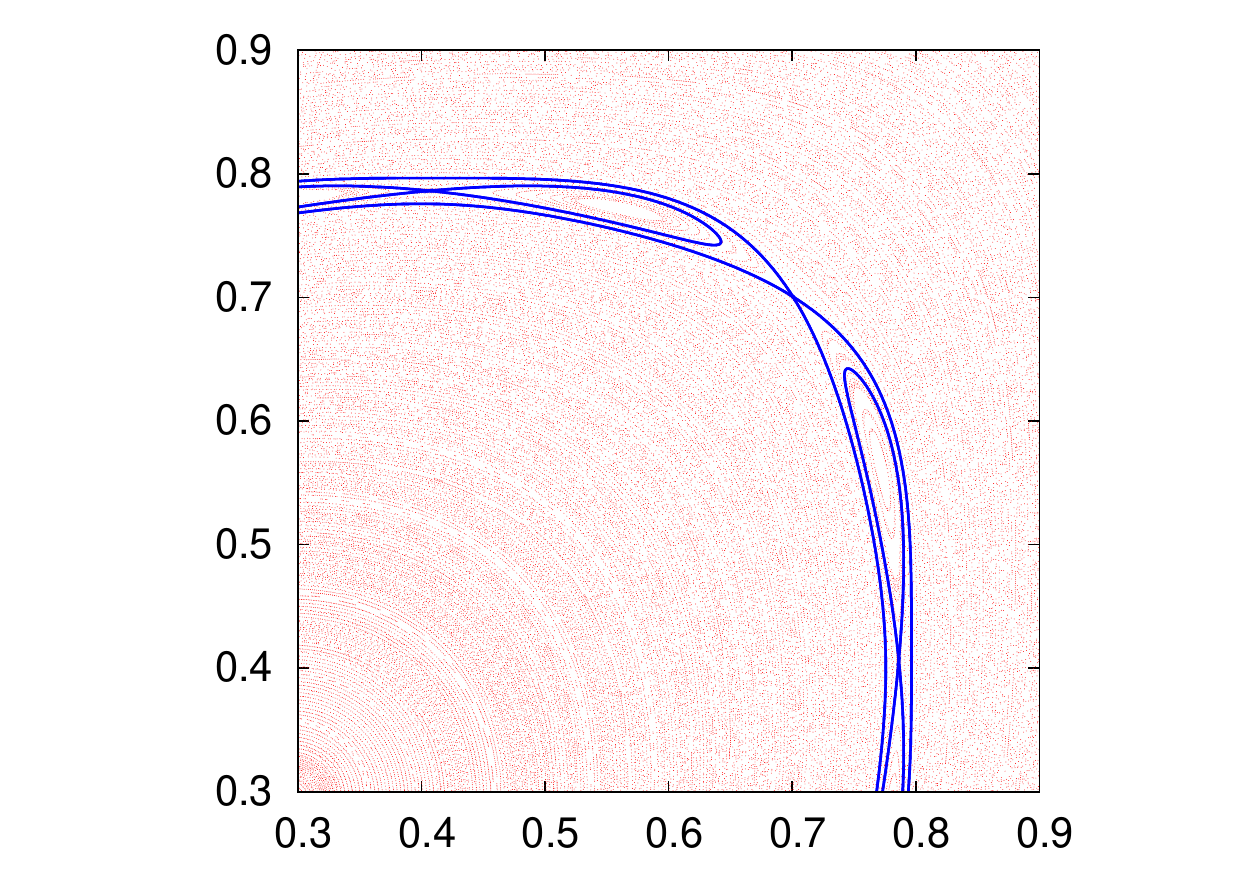}
        \hspace{-2.2cm}
	\includegraphics[height=5cm]{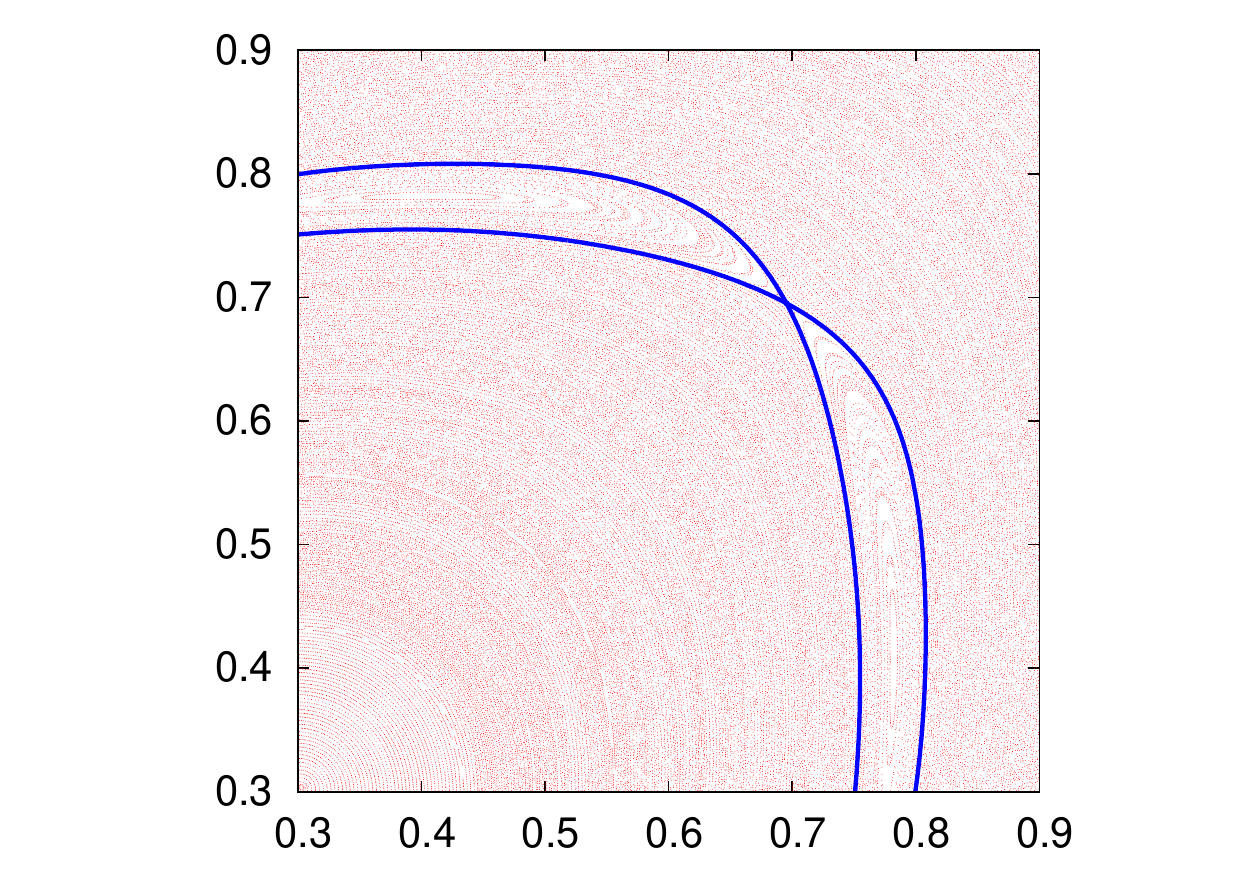}
	\caption{Sequence of bifurcations for fixed $M_2=-0.5$ in map
$\mathbf{C}_+$ when crossing $\hat B^r_4$ and $B^r_4$. The values of $M_1$
are: $0.7$, $0.715$, $0.718$ (top), $0.719$, $0.72$, $0.73$
(bottom).}
	\label{fig:res14_p_M20p5}
\end{figure}

\subsection{Other bifurcation curves related to 4-periodic orbits}

We have considered some of the local bifurcations of 4-periodic orbits
emanating from the 1:4 resonance. The evolution in phase space of the
associated 4-periodic orbits leads to other bifurcation curves of other
4-periodic orbits which interact with the obtained ones. In this section we aim
to illustrate some aspects of their configuration in parameter space. 

\begin{remark}
The curves $B_{4}^{l,r}$ and $\hat{B}_4^{l,r}$ are related to the local aspects
of the 1:4 resonance. On the other hand, the bifurcations curves considered in
this section are not related to the (local) 1:4 resonance problem. 
\end{remark}

Some of the bifurcation curves to be considered already appear in
Fig.~\ref{fig:res1_4p}.  The curves $C^{r,l}_4$ correspond to trace$=2$ (double eigenvalue 1),
given by
\begin{equation} \label{eqClr}
C^{r,l}_4:  \displaystyle 27 M_1^2 = -4 (1 + M_2)^3, \quad M_2 \le -1,
\end{equation}
and the curves $\tilde C^{r,l}_4$ to trace$=-2$ (double eigenvalue $-1$).
To explain the bifurcations that take place let us consider a horizontal line
$M_2=C$, with $C_1<C<C_2$, where $C_1 \approx -1.6220$ is the $M_2$-coordinate
of the intersection point between $C_4^r$ and $\tilde C^r_4$ (there is only one 
intersection point, shown in Fig.~\ref{fig:res1_4p} left) and $C_2\approx -1.0647$ is the
$M_2$-coordinate of the bottom intersection point between $\tilde C^{r}_4$ and
$\tilde C^{l}_4$ shown in the Fig.~\ref{fig:res1_4p} right.  When varying $M_1$ from right
to left in Fig.~\ref{fig:res1_4p} one has the following bifurcation
sequence:
\begin{itemize}
\item For $M_1$ in the right hand side of $\hat B^r_4$ there are saddle
4-periodic orbits along with a pair of elliptic 4-periodic orbits.  Recall that
curve $\hat B^r_4$ corresponds to a pitchfork bifurcation, hence to the
left of this curve we get a 4-periodic island of stability creating a garland
containing saddle and elliptic 4-periodic orbits. 

\item When decreasing $M_1$, an inverse period-doubling bifurcation occurs at
the first crossing with curve $\tilde C^r_4$ in Fig.~\ref{fig:res1_4p}
left. The bifurcation is as follows: for parameters to the right of $\tilde
C^r_4$, in a neighborhood of the elliptic 4-periodic orbit, there appears a
saddle 8-periodic orbit. This 8-periodic orbit bifurcates from the  saddle
4-periodic orbit that remains to the left of $\tilde C^r_4$.

\item For parameters on the curve $C^r_4$, a parabolic
bifurcation for a 4-periodic orbit takes place. At this bifurcation 
a saddle and an elliptic 4-periodic orbits are created. One of the pairs of the
elliptic and hyperbolic 4-periodic orbits that bifurcate lie on the symmetry
line $y=x$. The elliptic orbit undergoes a period-doubling bifurcation when
crossing the curve $\tilde C^r_4$. See Fig.~\ref{fig:res1_4p} right.
\end{itemize}

We have found other bifurcation curves related to 4-periodic orbits. We note
that the corresponding 4-periodic orbits do not lie on the symmetry line $y=x$
of map $\mathbf{C}_+$.  The bifurcation curves are shown in
Fig.~\ref{fig:res1_4p_down} left.
Note that, in the left plot, curves $D_4^{l,r}$, which wrap the full structure
of bifurcation curves shown, are confined in the region below the curve
$L^+_{\pi/2}$ partially shown in Fig.~\ref{fig:res1_4p}. Let us give some
details on the bifurcation curves obtained.

\begin{itemize}
\item The curves $D_4^{l,r}$ are given by 
\[D_4^{l,r}: \pm M_1 = \sqrt{\frac{ - M_2}{3}} \left( \frac{2}{3} M_2 +1 \right) \pm \sqrt{ - 3 M_2 -8}, \] 
and they correspond to period-doubling bifurcation curves of 2-periodic orbits.
The bifurcation curves related to 2-periodic orbits were studied in
\cite{GGO17}, in particular, curves $D_4^{l,r}$ were obtained there (they
correspond to the ones denoted by $L_2^{-1}$ and $L_2^{-2}$ in \cite{GGO17}).  

\item The bifurcation curves $\hat D_4^{l,r}$ correspond to a parabolic bifurcation of 
4-periodic orbits. The curve $\hat D_4^{l}$ ends up in a point $(M_1,M_2)=
(M_1^*,M_2^*) \approx (0.041064, -2.944529)$ where it becomes tangent to the
curve $D_4^{r}$.  See details in Fig.~\ref{fig:res1_4p_down} right.  When
$(M_1,M_2) \in \hat D_4^{l}$ tends to $(M_1^*,M_2^*)$ the 4-periodic orbit
approaches the 2-periodic orbit that undergoes the period-doubling bifurcation
in $D_4^{r}$.

\item The bifurcation curves $\bar D_4^{l,r}$, whose equations are
\begin{equation} \label{eqbD4lr}
\bar{D}_4^{l,r}:  27 M_1^2  \! = \!  \left(  3 \sqrt{3} + 3 \sqrt{-M_2} + 2 M_2 \sqrt{-M_2}\right)^2, \;\; M_2 \leq -3, 
\end{equation}
correspond to a pitchfork bifurcation of 4-periodic orbits. The curve $\bar
D_4^l$ ends up at the point $(M_1,M_2)=(0,-3)$ where curves $D_4^{l}$ and
$D_4^{r}$ intersect.

\item Finally, the bifurcation curves $\tilde D_4^{l,r}$ correspond to
4-periodic orbits with double eigenvalue $-1$, hence they are period-doubling
bifurcation curves of 4-periodic orbits.
\end{itemize}

\begin{figure}[tb]
%\psfrag{tD4l}{$\tilde D_4^l$}
%\psfrag{tD4r}{$\tilde D_4^r$}
%\psfrag{ D4r}{$D_4^r$}
%\psfrag{ D4l}{$D_4^l$}
%\psfrag{hD4r}{$\hat D_4^r$}
%\psfrag{hD4l}{$\hat D_4^l$}
%\psfrag{bD4r}{$\bar D_4^r$}
%\psfrag{bD4l}{$\bar D_4^l$}
%\includegraphics[width=0.45\textwidth]{bif_plus_below.eps}
%\includegraphics[width=0.45\textwidth]{bif_plus_below_dd.eps}

\hspace{0.3cm}
\includegraphics[]{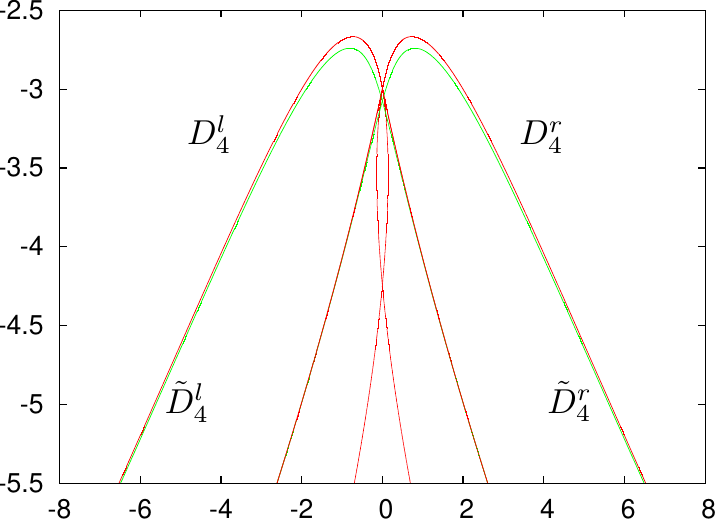}
\hspace{1cm}
\includegraphics[]{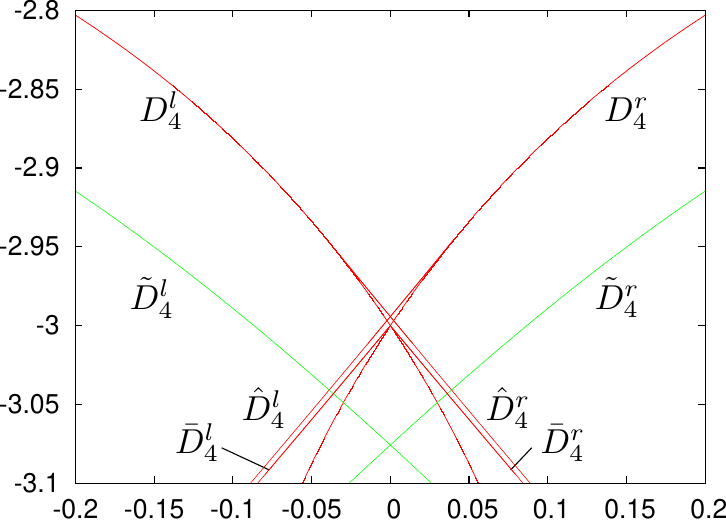}
\caption{Other bifurcation curves of 4-periodic orbits in map
$\mathbf{C}_+$. In the left plot we see the two (green) curves $\tilde D_4^{l,r}$ that correspond to
parabolic 4-periodic orbits with double eigenvalue $-1$. The curves
$D_4^{l,r}$ correspond to period-doubling of 2-periodic orbits.  The other
curves are better seen in the right plot, which is a magnification of the left
one. In the right plot we see that curves $\hat D_4^{l,r}$, corresponding
to parabolic 4-periodic with double eigenvalue $1$ respectively, become tangent
to $D_4^{l,r}$. Finally, curves $\bar D_4^{l,r}$, which correspond to a
pitchfork bifurcation of 4-periodic orbits, end up at $(M_1,M_2)=(0,-3)$.} 
\label{fig:res1_4p_down}
\end{figure}

To provide further details of the sequence of bifurcations that take place we consider
the horizontal line $M_2=-3.1$ and, on this line, different values of $M_1<0$.
We refer to Fig.~\ref{fig:res1_4p_down} right to see the location of the
different bifurcation curves. To start with we consider $M_1 = -0.047$, which
is between the curves $D_4^{r}$ and $\tilde D_4^{r}$. For parameter values
$(M_1,M_2)=(-0.047,-3.1)$ the phase space shows up a 2-periodic island and a
4-periodic island (which is a 2-periodic satellite of the 2-periodic island).
Denote by $e_2$ (resp. by $e_4$) the elliptic points around which the
2-periodic (resp. the 4-periodic) islands of stability are organized. The
2-periodic elliptic island, for $M_2=-3.1$, persists\footnote{If one considers
$M_2<4$, the 2-periodic elliptic point persists when moving $M_1$ from right to
left until crossing a parabolic bifurcation curve of 2-periodic orbits. This
curve was obtained in \cite{GGO17} and was denoted there by $L_2^+$.} for
$-0.047< M_1 \leq 0$. When changing $M_1$ on the line $M_2=-3.1$, the following
bifurcations are observed:
\begin{itemize}
\item At the crossing of $\tilde D_4^{r}$ from left to right, the point $e_4$
undergoes a period-doubling bifurcation, and an 8-periodic elliptic point is
created.
\item When crossing $D_4^r$ from right to left, the point $e_2$ undergoes a
period-doubling. Consequently, a 4-periodic elliptic orbit is born, denote it
by $\tilde{e}_4$.
\item At the crossing of $\bar{D}_4^l$ from right to left, there is an inverse
pitchfork bifurcation at which the two 4-periodic elliptic points $e_4$ and
$\tilde{e}_4$ collide and give rise to a 4-periodic elliptic orbit. This
4-periodic elliptic orbit persists until the crossing of $\hat{D}_4^l$ where
disappears at a parabolic bifurcation.
\end{itemize}

Other sequences of bifurcations can be observed for other lines $M_2=C$ (specially
when considering $M_2$ in the range shown in Fig.~\ref{fig:res1_4p_down} right).
For example, for $M_2 = -2.9$ and moving $M_1$ from right to left starting from $M_1=0$, 
one has that the 2-periodic elliptic orbit
undergoes a period-doubling bifurcation at $D_4^l$ (roughly for $M_1 \approx
-0.1$). The 4-periodic elliptic orbit that bifurcates from the previous
bifurcation undergoes a period-doubling at $\tilde D_4^l$ (which happens for
$M_1 < -0.2$). 

As a final comment, we note that the bifurcation curves considered before allow
us to explain the sequences of bifurcations of 4-periodic orbits that we have
observed when plotting the islands of stability for different values of
$(M_1,M_2)$. Hence, we believe that the bifurcation diagram for the 4-periodic
orbits in the $(M_1,M_2)$ regions shown is complete (although we have no proof
of this fact).

\section{Conclusions and related topics} \label{sec:conclusions}

We have obtained a detailed picture of the bifurcation diagrams near the 1:4
resonance of maps $\mathbf{C}_{\pm}$ in (\ref{cubH1m})-(\ref{cubH1pl}). A
description of the bifurcations taking place when crossing the main bifurcation
curves derived (either analytically or numerically) has been provided.
Special emphasis has been put to the clarification of the scenarios related to the
degeneracies of the 1:4 normal form. We have shown that degeneracy $A=1$ happens
for $\mathbf{C}_-$ while degeneracy $B_{03}=0$ happens for $\mathbf{C}_+$,
and we have analysed the dynamical consequences of these degeneracies in these
concrete cases.

We believe that the results presented in this work are relevant for related
studies. Namely, we want to emphasize that the study of the cubic H\'enon
maps (\ref{cubH1m}) and (\ref{cubH1pl}) (and naturally the quadratic one
(\ref{H2st})) is important because 
\begin{itemize}
\item these maps are the simplest nonlinear symplectic maps of the plane, and
therefore the understanding of the basic properties of the dynamics and the bifurcations
in such maps will be very useful in more general contexts.
\item as pointed out in the introduction, these maps are, in fact, normal forms
of the first return maps near the quadratic (map (\ref{H2st})) and cubic (maps
(\ref{cubH1m}) and (\ref{cubH1pl})) homoclinic tangencies:
it is easy to relate the structure of the bifurcations of these maps
with the global bifurcations happening at the homoclinic tangency.
\end{itemize}

However, there is another important reason why the 1:4 resonance in the cubic
H\'enon maps is of interest. It is connected with mixed dynamics -- a new
third type of dynamical chaos characterized by the principal inseparability of
attractors from repellers\footnote{Here attractor and repeller are considered in the
Conley-Ruelle sense \cite{Con78, Rue81}, see also \cite{GT17}.} and from the 
conservative elements of dynamics (for example, periodic sinks, sources and
elliptic points), see e.g.  \cite{GST97,LS04,DGGLS13,G16,GT17}. It is worth
noting that the mixed dynamics
can be an open property of reversible non-conservative chaotic systems in which
self-symmetric orbits are conservative (e.g., symmetric elliptic trajectories),
while asymmetric ones appear in pairs and have the opposite type of stability
(e.g., ``sink-source'' pairs).  Such symmetric/asymmetric orbits emerge usually
as a result of various homoclinic bifurcations, see more details in
\cite{DGGLS13}, including local symmetry breaking  bifurcations like reversible
pitchfork ones \cite{LT12}. However, the structure of such
bifurcations in many cases is not known, as happens for example in the case of
symmetric cubic homoclinic tangencies. It can be deduced from \cite{GGO17} that
the first return map near a symmetric cubic homoclinic tangency must coincide
in the main order with the cubic H\'enon map either of form (\ref{cubH1m}) or
(\ref{cubH1pl}) which are reversible maps. When studying the problem of 1:4
resonance in these maps, we have shown that pitchfork bifurcations appear
accompanying the resonance local bifurcation.  These bifurcations should lead
to the birth of a ``sink-source'' pair of periodic orbits in general reversible
contexts.  

We believe that these topics certainly deserve future devoted studies and we
hope that the results presented here will contribute to facilitate them.

\section*{Acknowledgments}

The paper is carried at the financial support of the RSF grant 14-41-00044.  AV
and MG has been supported by the Spanish grant MTM2016-80117-P (MINECO/FEDER,
UE). AV also thanks the Catalan grant 2014-SGR-1145.  MG also thanks the Juan
de la Cierva-Formaci\'on Fellowship FJCI-2014-21229 and the MICIIN/FEDER grant
MTM2015-65715-P. 
SG thanks the Russian Foundation for Basic Research, grant
16-01-00364, and the Russian Ministry of Science and Education, project
1.3287.2017 -- target part, for supporting the scientific research.
IO points out that this paper is a contribution to the project M2 (Systematic
multiscale modelling and analysis for geophysical flow) of the Collaborative
Research Centre TRR 181 ``Energy Transfer in Atmosphere and Ocean'' funded by the
German Research Foundation.

\appendix

\section{Bifurcation curves associated with parabolic 4-periodic orbits in
$\mathbf{C}_\pm$}\label{sec:4par}

In this section we derive some of the equations of the bifurcation curves
displayed in Fig.~\ref{fg:1p4menys_corves},~\ref{fig:res1_4p}
and~\ref{fig:res1_4p_down}.  These curves correspond to the appearance of
parabolic 4-periodic orbits in maps~(\ref{cubH1m}) and~(\ref{cubH1pl}). 

The curves corresponding to trace$=2$ are given by the following lemma.

\begin{lemma}\label{lm:4par_minus}
The following bifurcation curves corresponding to parabolic 4-periodic orbits
with double eigenvalue $1$ exist:
\begin{enumerate}
\item For map $\mathbf{C}_-$ $($\ref{cubH1m}$)$, curves $L^i_4$,
$i=1,2,3,4$, displayed in Fig.~\ref{fg:1p4menys_corves}, are given by the
equations $($\ref{eqL412}$)$, $($\ref{eqL43}$)$ and $($\ref{eqL44}$)$.

\item For map $\mathbf{C}_+$ $($\ref{cubH1pl}$)$, curves $B^{l,r}_4$, $\hat
B^{l,r}_4$ and $C^{l,r}_4$, displayed in Fig.~\ref{fig:res1_4p}. They are
given by $($\ref{eqBlrhBlr}$)$ and $($\ref{eqClr}$)$.
Moreover, the curves $\bar D^{l,r}_4$, displayed in
Fig.~\ref{fig:res1_4p_down}, are given by $($\ref{eqbD4lr}$)$.  Also, curves
$\hat D^{l,r}_4$, displayed in Fig.~\ref{fig:res1_4p_down}, satisfy the
relations $($\ref{p4_21}$)$.
\end{enumerate}
\end{lemma}

\begin{proof}
	We rewrite maps~(\ref{cubH1m}) and~(\ref{cubH1pl}) in the form
$$	\begin{pmatrix} \bar x \\ \bar y \end{pmatrix} =
	\mathbf{C}_\pm \begin{pmatrix} x \\ y \end{pmatrix} =
	\left( \begin{matrix}
	y \\ -x + P(y)
	\end{matrix} \right).
	$$
where $P(y) = M_1 + M_2 y - \delta y^3$, being $\delta=1$ in 
case of map (\ref{cubH1m}) and $\delta=-1$ in case of map (\ref{cubH1pl}).  The point $Q$
is a parabolic 4-periodic orbit with trace $2$ for $\mathbf{C}_{\pm}$ if the
following conditions are satisfied: 
\begin{equation}\label{eq:par4_cond}
\text{(a) } \mathbf{C}_{\pm}^4(Q) = Q \text{ and (b) } \mbox{tr} \;
D(\mathbf{C}_{\pm}^4(Q)) = 2,
\end{equation}
where $D$ stands for the Jacobi matrix.
	
Condition (a) is obviously equivalent to 
$\mathbf{C}_{\pm}^{-2}(Q) =\mathbf{C}_{\pm} ^2(Q)$,
which yields
	\begin{equation} \label{p4_5}
	(y^2 - y P(x) + P^2(x) - \delta M_2) (-2y + P(x)) = 0,  \;\;\;
	(x^2 - x P(y) + P^2(y) - \delta M_2) (-2x + P(y)) = 0,
	\end{equation}
where $(x,y)$ are the coordinates of the point $Q$. It is easy to check that if
$(-2y + P(x))=0$ and $(-2x + P(y))=0$ simultaneously, $Q$ is actually either a
2-periodic orbit or a fixed point. For this reason, for 4-periodic orbits, we
assume that at least one of these expressions is non-zero. Then, we get 
two different cases, which we consider separately: 
\begin{itemize}
\item[] Case 1. $(-2y + P(x)) \neq 0$ and $(-2x + P(y)) \neq 0$;

\item[] Case 2. $(-2y + P(x)) = 0$, $(-2x + P(y)) \neq 0$ (the case when $(-2x + P(y)) = 0$,  $(-2y + P(x)) \neq 0$ is considered analogously).
\end{itemize}

In Case 1, equations~(\ref{p4_5}) are rewritten as follows 
	\begin{equation} \label{p4_6}
	y^2 - y P(x) + P^2(x) = \delta M_2, \;\;\; 
	x^2 - x P(y) + P^2(y) = \delta M_2.
	\end{equation}
	
The last equations show that pairs $(y, P(x))$ and $(x, P(y))$ are the
coordinates of points on the ellipse given by equation $X^2 - X Y + Y^2 =
\delta M_2$, and that~(\ref{p4_6}) has solutions for~$M_2 \ge 0$ in case
of~$\mathbf{C}_{-}$ and for~$M_2 \le 0$ in case of~$\mathbf{C}_{+}$. Thus,
we introduce the parametrization with parameters~$t_1$ and~$t_2$ along this
ellipse in such a way that for some $-\pi \le t_1, t_2 \le \pi$ we have:
	\begin{equation} \label{p4_7}
	\begin{matrix}
	y = \sqrt{\delta \frac{M_2}{3}} \cos{t_1} - \sqrt{\delta M_2} \sin{t_1}, & x = \sqrt{\delta \frac{M_2}{3}} \cos{t_2} - \sqrt{\delta M_2} \sin{t_2}, \\
	P(x) = 2 \sqrt{\delta \frac{M_2}{3}} \cos{t_1}, &  P(y) = 2 \sqrt{\delta \frac{M_2}{3}} \cos{t_2}.
	\end{matrix}
	\end{equation}
	
The values of~$x, y, P(x), P(y)$ given by the parametrization~(\ref{p4_7}) satisfy
	\begin{equation} \label{p4_8}
	\begin{split}
	2 \sqrt{\delta \frac{M_2}{3}} \cos{t_1} &= P(x) = P \left(\sqrt{\delta\frac{M_2}{3}} \cos{t_2} - \sqrt{\delta M_2} \sin{t_2}\right) =
	M_1 + \frac{2}{3 \sqrt{3}} M_2 \sqrt{\delta M_2} \cos 3 t_2,  \\
	2 \sqrt{\delta \frac{M_2}{3}} \cos{t_2} &= P(y) = P \left(\sqrt{\delta \frac{M_2}{3}} \cos{t_1} - \sqrt{\delta M_2} \sin{t_1}\right) =
	M_1 + \frac{2}{3 \sqrt{3}} M_2 \sqrt{\delta M_2} \cos 3 t_1 .
	\end{split}
	\end{equation}

Condition~(b) in~(\ref{eq:par4_cond}) gives us the equation
	\begin{equation} \label{p4_10}
	\left[P'(y) +  P'(-y + P(x))\right]  \left[P'(x) +  P'(-x + P(y))\right] = P'(y) P'(-y + P(x)) P'(x) P'(-x + P(y)),
	\end{equation}
	which, using the parametrization~(\ref{p4_7}), is rewritten as
	\begin{equation} \label{p4_11}
	16 M_2^2 \sin^2 t_1 \sin^2 t_2 \left[ 1- M_2^2 (1 + 2 \cos 2 t_1) (1 + 2 \cos 2 t_2) \right]=0,
	\end{equation}
	therefore either $|\cos t_{1,2}| > 1/2$ or $|\cos t_{1,2}| < 1/2$.

Thus, we get 3 equations, (\ref{p4_8}) and (\ref{p4_11}), for variables $t_1,
t_2, M_1, M_2$ and, in order to obtain the desired bifurcation curves, we just
need to exclude $t_1$ and $t_2$.
	
Consider first Case 1.1 when $\cos t_1 = \cos t_2$. With this, equations
(\ref{p4_8}) completely coincide and the parametric equation of the bifurcation
curve is the following
$$
M_1 = \frac{2}{3\sqrt{3}} \sqrt{\delta M_2} (3 \cos t_1 - M_2 \cos 3 t_1), \;\;\; M_2 = \frac{\delta}{|1 +  2 \cos 2 t_1|}, \;\; 0 < t_1 < \pi.
$$ 

Excluding  $\cos t_1$, we obtain the explicit formulas of the curves, namely:
in the case~$\mathbf{C}_{-}$, we get $L_4^1$ for $0 < t_1 < \pi / 3$,
$L_4^3$ for $\pi / 3 < t_1 < 2 \pi / 3$ and $L_4^2$ for $2 \pi / 3 < t_1 <
\pi$, while in the case~$\mathbf{C}_{+}$, we get $\hat B^r_4$ for $0 < t_1 <
\pi / 3$, $C^{r,l}_4$ for $\pi / 3 < t_1 < 2 \pi / 3$ and $\hat B^l_4$ for $2
\pi / 3 < t_1 < \pi$.
	
Consider next Case 1.2 when $\cos t_1 \neq \cos t_2$, then equations
(\ref{p4_8}) can be solved with respect to $M_1$ and $M_2$:
\begin{equation} \label{p4_14}
M_1 = \frac{2}{3\sqrt{3}} \sqrt{\delta M_2} (3 \cos t_2 - M_2 \cos 3 t_1), \;\; M_2 = \frac{3 (\cos t_1 - \cos t_2)}{\cos 3 t_2 - \cos 3 t_1}.
\end{equation}

Plugging $M_2$ into equation~(\ref{p4_11}) we obtain the following relation between $t_1$ and $t_2$
$$
4 \cos^2 t_1  + 16 \cos t_1 \cos t_2 + 4 \cos^2 t_2 + 3 = 0,
$$
where one can see that variables $\cos t_1$ and $\cos t_2$ satisfy the equation
of a hyperbola $4 X^2 + 16 X Y + 4 Y^2 + 3 = 0$. Thus, they can be parametrized
in the following way:
$$
\begin{array}{l}
\displaystyle \cos t_1 = \frac{\sqrt{2}}{8}(1 - \sqrt{3})t - \frac{\sqrt{2}}{8}(1 + \sqrt{3}) \frac{1}{t}, \;\:\;
\displaystyle \cos t_2 = \frac{\sqrt{2}}{8}(1 + \sqrt{3})t - \frac{\sqrt{2}}{8}(1 - \sqrt{3}) \frac{1}{t}. %\\
\end{array}
$$
The natural conditions $|\cos t_{1,2}| \le 1$ are fulfilled only for $t_- \le
|t| \le t_+$, where $\displaystyle t_\pm = \frac{\sqrt{3} \pm 1}{\sqrt{2}}$.
Substituting this parametrization into~(\ref{p4_14}) and excluding $t$ gives
the curve $L^4_4$ in the case~$\mathbf{C}_{-}$. Note that for this
parametrization, we have $\displaystyle M_2 = -\frac{4 t^2}{t^4 - 4 t^2 + 1}$
which takes only positive values for $t_- \le |t| \le t_+$, and this case
does not provide any bifurcation curve for~$\mathbf{C}_{+}$.

In Case 2 we get the equations $ 2y = P(x)$, $x^2 - x P(y) + P^2(y) =
\delta M_2$ and the equation~(\ref{p4_10}). The second one admits the
parametrization as before (note that $t \neq n \pi$, otherwise we have $2x =
P(y)$, i.e. a 2-periodic orbit or a fixed point):
$$
x = \sqrt{\delta\frac{M_2}{3}} \cos{t} - \sqrt{\delta M_2} \sin{t}, \;\; P(y) = 2 \sqrt{\delta \frac{M_2}{3}} \cos{t},
$$
and equation (\ref{p4_10}) is written as
\begin{equation} \label{p4_18}
2 (M_2 - 3 \delta y^2)(-4 M_2 \sin^2 t) = (M_2 - 3 \delta y^2)^2 (- 4 M_2^2 \sin^2 t) (1 + 2 \cos 2 t).
\end{equation}

Let us consider Case 2.1, when $M_2 - 3 \delta y^2=0$ in~(\ref{p4_18}),
i.e. $\displaystyle y = \pm \sqrt{\delta\frac{M_2}{3}}$.  Then we have
$$
\begin{array}{l}
\displaystyle 2 \sqrt{\delta\frac{M_2}{3}} \cos{t} = P(y) = M_1 \pm \frac{2}{3 \sqrt{3}} M_2 \sqrt{\delta M_2}, \\
\displaystyle \pm 2 \sqrt{\delta\frac{M_2}{3}} = P(x) = M_1 + \frac{2}{3 \sqrt{3}} M_2 \sqrt{\delta M_2} \cos 3 t.
\end{array}
$$

The possible bifurcation curves are given parametrically in the following form
\begin{equation} \label{p4_20}
\displaystyle M_2 = -\frac{3}{(2 \cos t \pm 1)^2}, \;\; M_1 = 2 \sqrt{\delta\frac{M_2}{3}} \cos{t} \mp \frac{2}{3 \sqrt{3}} M_2 \sqrt{\delta M_2}.
\end{equation}

Since $M_2<0$, the equations~(\ref{p4_20}) do not give any curve for the
map~$\mathbf{C}_-$. In the case of~$\mathbf{C}_+$ we obtain the equations for
curves $B^{l,r}_4$ (which corresponds to the branch that is obtained for $t
\in [0,2\pi/3)$) and the equations for the curves
$\bar{D}_4^{l,r}$ (for $t\in (2\pi/3,\pi]$).

For the Case 2.2, when $M_2 - 3 \delta y^2 \neq 0$ in~(\ref{p4_18}), we have the equations
\begin{equation}\label{p4_21}
\begin{array}{l}
2  = (M_2 - 3 \delta y^2) M_2 (1 + 2 \cos 2 t), \\
\displaystyle 2 \sqrt{\delta \frac{M_2}{3}} \cos{t} = P(y) = M_1 + M_2 y -\delta y^3, \\
\displaystyle M_1 + \frac{2}{3 \sqrt{3}} M_2 \sqrt{\delta M_2} \cos 3 t = P(x) = 2 y.       
\end{array}
\end{equation}
We have numerically checked that $\hat{D}_4^{l,r}$ fulfill these implicit equations.  
\end{proof}

Concerning parabolic 4-periodic orbits with double eigenvalue $-1$ (trace$=-2$), an analogous proof to the previous Lemma can be done. The only difference that condition~(b) in~(\ref{eq:par4_cond}) now is~$\mbox{tr} \; D(\mathbf{C}_{\pm}^4(Q)) = -2$ which gives instead of~(\ref{p4_10}) and~(\ref{p4_11}) the third equation
\begin{equation} \label{p4_10m}
\left[P'(y) +  P'(-y + P(x))\right]  \left[P'(x) +  P'(-x + P(y))\right] = %\\
%&
P'(y) P'(-y + P(x)) P'(x) P'(-x + P(y))+4,
\end{equation}
that, using the parametrization~(\ref{p4_7}) in Case 1, is rewritten as
\begin{equation} \label{p4trm2_1}
16 M_2^2 \sin^2 t_1 \sin^2 t_2 \left[ 1- M_2^2 (1 + 2 \cos 2 t_1) (1 + 2 \cos 2 t_2) \right]=4.
\end{equation}

Thus, the bifurcation curves $\tilde L^i_4$, $i=1,2,3,4$, $\tilde C^{l,r}_4$
and $\tilde D^{l,r}_4$, displayed in
Figs.~\ref{fg:1p4menys_corves},~\ref{fig:res1_4p} and~\ref{fig:res1_4p_down},
respectively, should satisfy the equations~(\ref{p4_8}) and~(\ref{p4_10m}). For
example, in the simplest case $\cos t_1= \cos t_2$, which is analogous to Case 1.1
in the proof of Lemma~\ref{lm:4par_minus},  the equation~(\ref{p4trm2_1}) becomes
$$
4 M_2^2 \sin^4 t_1 \left[ 1- M_2^2 (1 + 2 \cos 2 t_1)^2 \right]=1,
$$	
which makes sense for $M_2$ to exist for $|\cos t_1| \leq \sqrt{2/5}$. Using the last equation as well as~(\ref{p4_8}), we get the parametric expression 
$$
M_1 = \frac{2}{3\sqrt{3}} \sqrt{\delta M_2} (3 \cos t_1 - M_2 \cos 3 t_1), \;\;\; M_2 = \delta \frac{\sqrt{\sin^2 t _1\pm \sqrt{\sin^4 t_1- (1+2\cos(2 t_1))^2}}}{2 |1 +  2 \cos 2 t_1| |\sin t_1|}, \;\; 
$$
where $\arccos\sqrt{{2}/{5}} \leq t_1 \leq
\arccos\left(-\sqrt{{2}/{5}}\right)$. The parametric expression gives the
curves $\tilde L^{1,2}_4$ (note that for $t_1=\pi/2$ the curves have the
intersection point $M_1=0, M_2=1/\sqrt{2}$), in the case of 
map~$\mathbf{C}_-$, and the curves $\tilde C^{l,r}_4$ (as well for $t_1=\pi/2$
the curves  intersect at $M_1=0, M_2=-1/\sqrt{2}$), in the case of 
map~$\mathbf{C}_+$.

The other cases (analogous to Cases 1.2, 2.1 and 2.2 in  the proof of Lemma~\ref{lm:4par_minus}) 
lead to cumbersome equations that we omit.

\section{Local analysis of the 1:4 resonance with degeneracy $B_{03}=0$} \label{appendix1p4deg} 

Consider a one-parameter family 
$F_{\delta}:\mathbb{R}^2 \rightarrow \mathbb{R}^2$ of area-preserving maps
which admit the reversibility $(x,y) \rightarrow (y,x)$. Let us consider $\delta \in \mathbb{R}$, small enough, 
being a parameter that unfolds a 1:4 resonant fixed point. Without loss of
generality, we assume $F_{\delta}(0)=0$ and $\text{Spec}DF_0(0)=\pm i$. There
exists a (formal) change of coordinates such that it reduces $F_{\delta}$ to its
Takens normal form $N_{\delta}$ (see \cite{Tak74,Bro90}) which commutes with
the linearized map
$\Lambda_0$ of $F_0$ at $0$, that is, 
\[ N_{\delta} \circ \Lambda_0 = \Lambda_0 \circ N_{\delta} \] 
The normal form is easily expressed in terms of complex conjugated variables $z=x+iy$, $z^*=x-iy$.
The map $\Lambda_0^{-1} N_{\delta}$ is
near-the-identity and can be (formally) interpolated by a Hamiltonian flow, that
is,
\[ \Lambda_0^{-1} N_{\delta} = \phi_{H_{\delta}}^{t=1}, \]
where the Hamiltonian function $H_{\delta}$ is $\Lambda_0$-invariant and is the sum of resonant terms
\[
H_{\delta} (z,z^*) =  \sum_{j-k \in \Gamma } h_{j,k} z^j (z^*)^k,
\]
being
\[
\Gamma = \{ s\in \mathbb{Z}, s =0 \text{ (mod 4)}\}
\]
the set of resonant monomials.

The area-preserving condition implies that coefficients $h_{i,i}$ are real. 
Moreover, the reversibility $(x,y) \rightarrow (y,x)$ imply that $h_{ij}=h_{ij}^*$.
Then, by introducing Poincar\'e polar (symplectic) coordinates 
\[ I = \frac{|z|^2}{2}, \qquad \varphi=\text{arg}(z), \]
the interpolating Hamiltonian is reduced
to 
\[
H(I,\varphi)= b_1 I + b_2 I^2 + b_3 I^3 + \mu I^2 \cos(4 \varphi) + B I^3 \cos(4 \varphi + \varphi_1) + \mathcal{O}(I^4),
\]
where all the coefficients are real and depend on $\delta$, and $\varphi_1$ is an initial phase.

We are interested in the analysis of the unfolding of the degenerate case (i.e.
the situation with $B_{03}=0$). This means that $b_1=\mu=0$ at the exact
resonance (when $\delta=0$). Note that it is natural to adjust the angle
variable so that $\varphi_1=0$. Taking into account that $b_1=\delta$ one has
\[
H(I,\varphi)= \delta I + b_2 I^2 + b_3 I^3 + \mu I^2 \cos(4 \varphi) + B I^3 \cos(4 \varphi) + \mathcal{O}(I^4).
\]

Our goal is to describe the bifurcations when $\delta,\mu \neq 0$ but small.
Note that this degenerate 1:4 resonance case leads to a codimension two
bifurcation problem. 

Ignoring higher order terms in $I$, the equations are
\begin{align}
\dot{I}&=-\frac{\partial{H}}{\partial{\varphi}} = 4 \mu I^2 \sin(4\varphi) + 4 B I^3 \sin(4 \varphi) + \mathcal{O}(I^4),\\
\dot{\varphi}&= \frac{\partial{H}}{\partial I}= \delta + 2 b_2 I + 3 b_3 I^2 + 2 \mu I \cos(4\varphi) + 3 B I^2 \cos(4 \varphi) + \mathcal{O}(I^3).
\end{align}
We look for fixed points $(I_*,\varphi_*)$  with $I_*>0$. Requiring $\dot{I}=0$ one obtains
\[ I_* = -\frac{\mu}{B} + 
\mathcal{O}(\mu^2),
\]
and substituting it into the equation $\dot{\varphi}=0$ gives
\[
(2 \mu I_* + 3 B I_*^2) \cos(4\varphi) = -(\delta+2b_2I_*) + \mathcal{O}(\mu^2), 
\]
which, fixed a small value of $\mu$, determines a small range of values of
$\delta$ (close to $-2b_2 I_*$) for which there are fixed points.  These are non-symmetric points
which are created and disappear in pitchfork bifurcations. See \cite{GLRT14} for
related comments.
One has 
\[ I_* = -\frac{\mu}{B}, \qquad \cos(4 \varphi_*) = - \frac{B}{\mu^2} \left(b_1 - 2 b_2 \frac{\mu}{B} \right). \]  
For illustrations, we consider $b_2=1$, $b_3=0$, $B=1$, $\mu=-0.15$.
The phase spaces for $\delta=-0.28,-0.29,-0.3,-0.31$ and $-0.32$ are shown in Fig.~\ref{pitchfork_deg}.

\begin{figure}[tb]
	\centering
        \hspace{-1cm}
	\includegraphics[height=5cm]{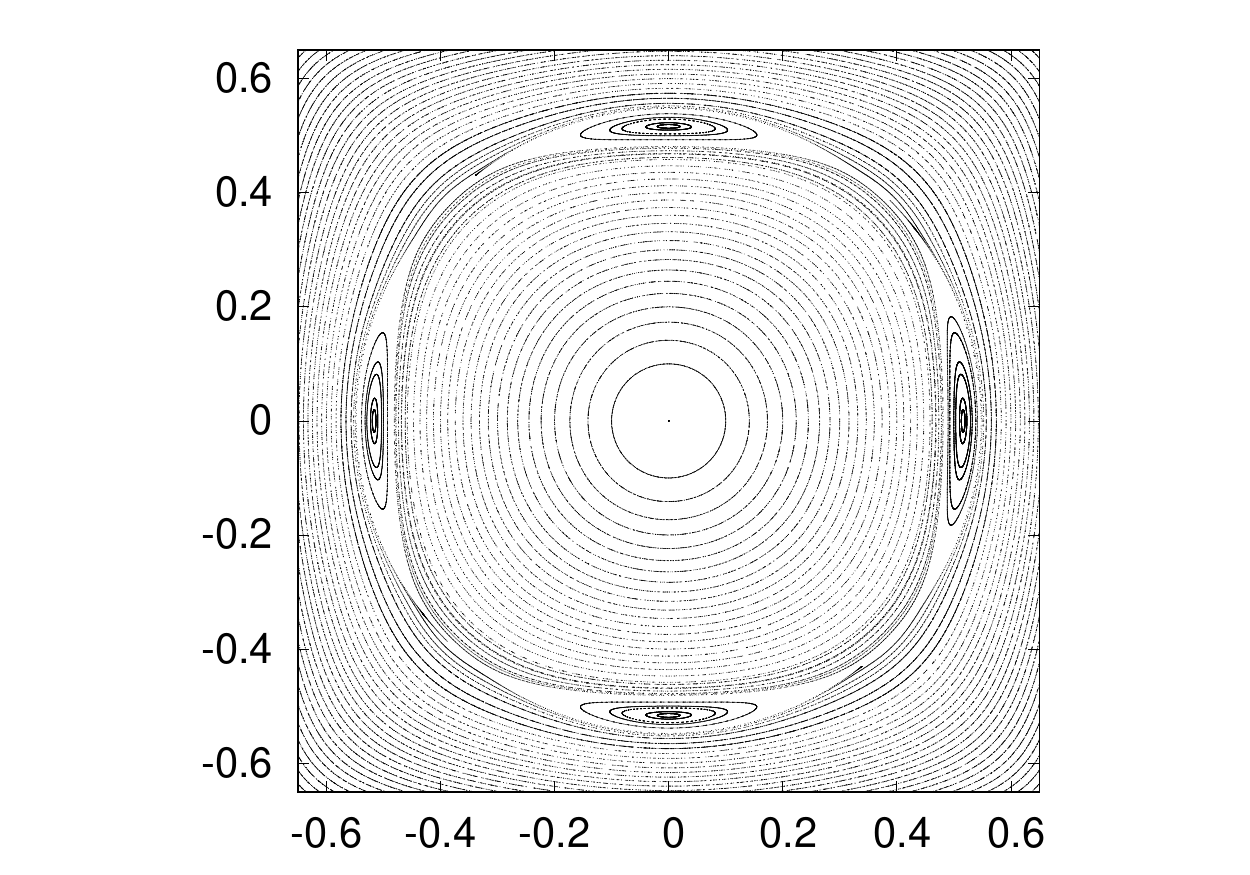}
        \hspace{-1.5cm}
	\includegraphics[height=5cm]{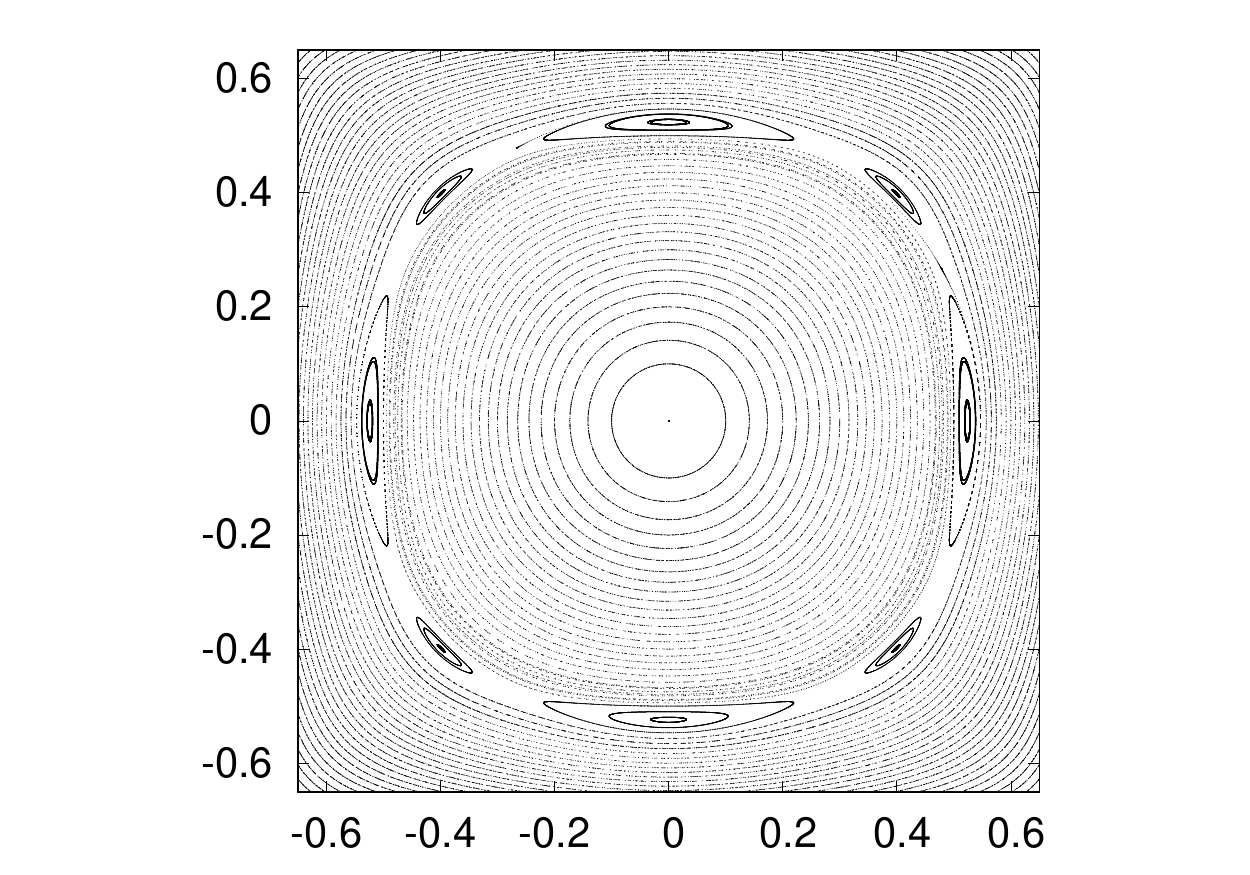}
        \hspace{-1.5cm}
	\includegraphics[height=5cm]{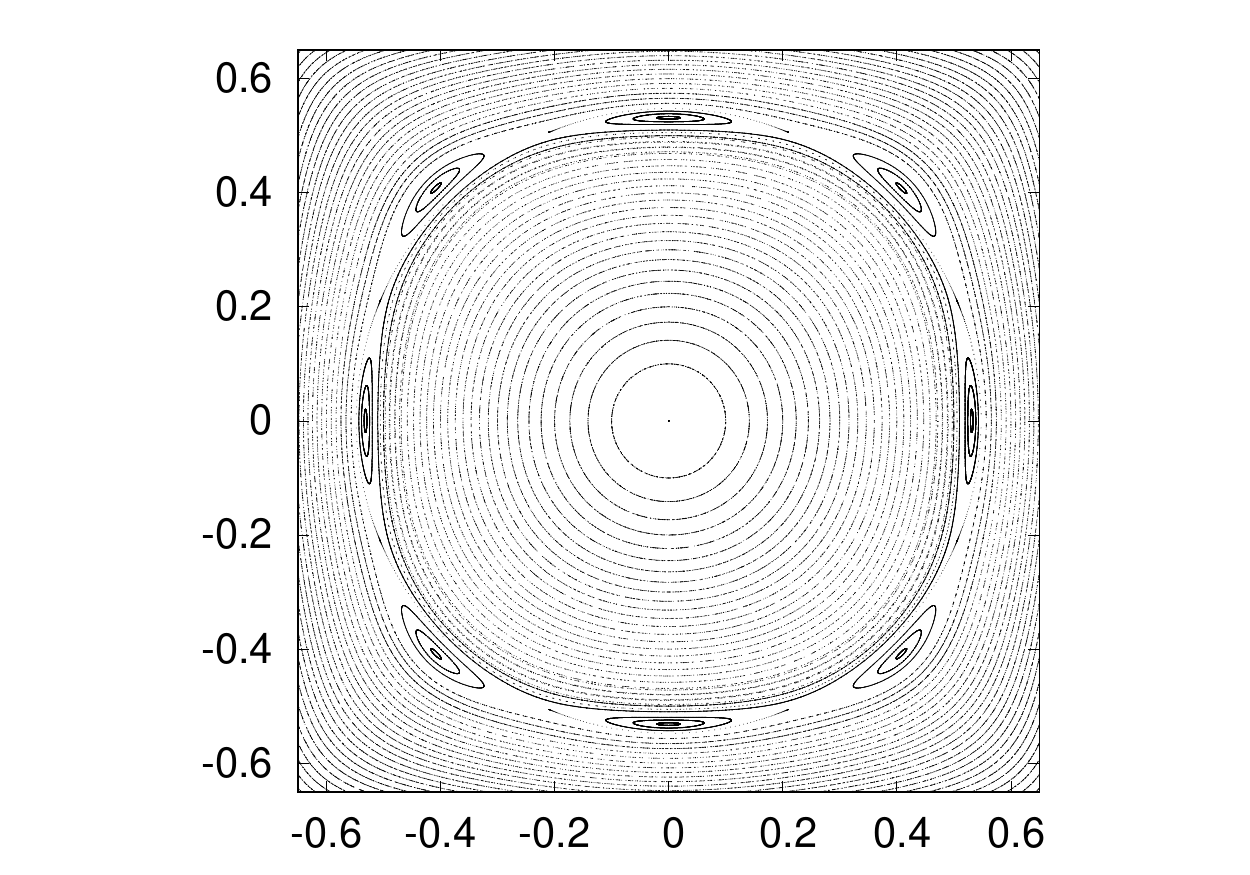} \\
	\includegraphics[height=5cm]{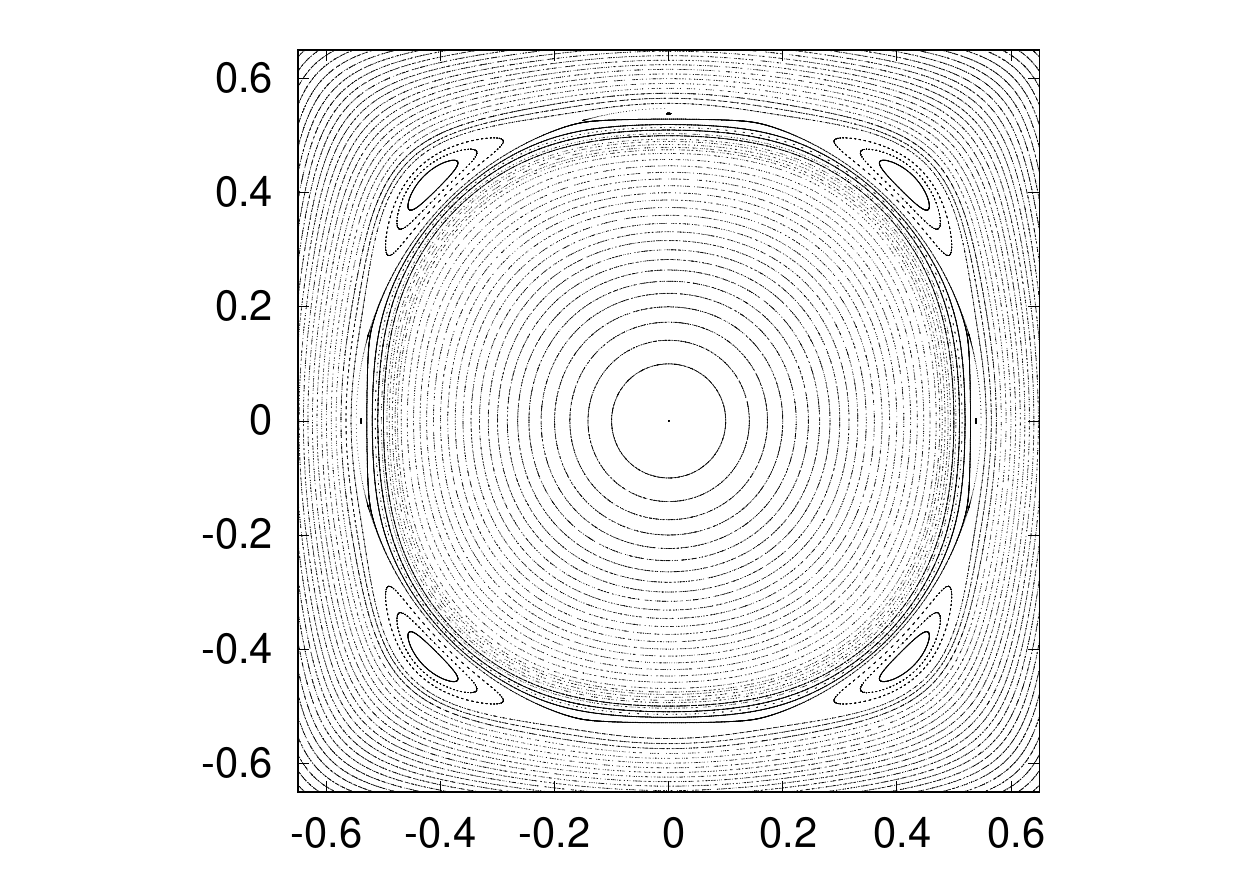}
	\includegraphics[height=5cm]{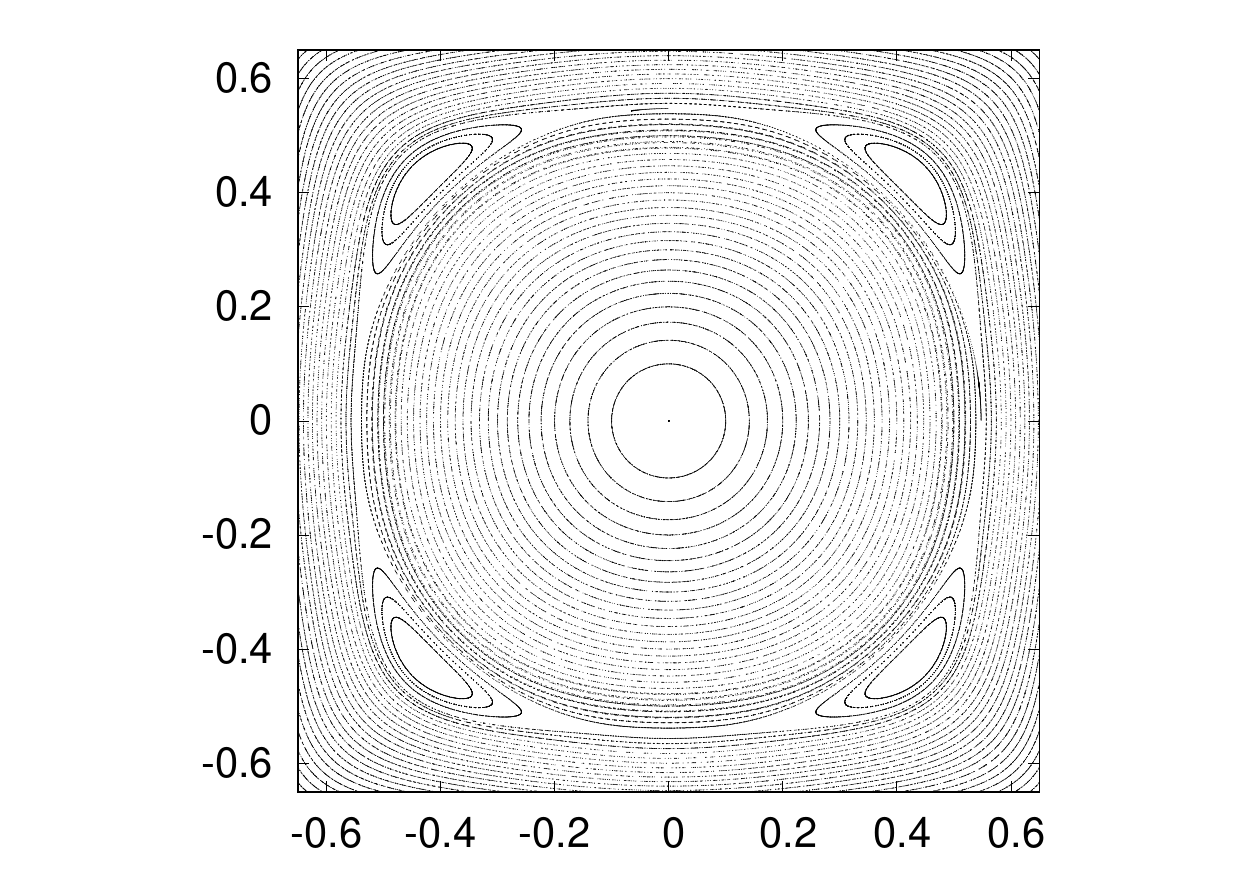}
	\caption{{ We show the phase space for $\delta=-0.28, \ -0.29, \ -0.3$ (first row) and $\delta = -0.31, \ -0.32$ (second row)
			for the selected values of the parameters, see text for details.}}
	\label{pitchfork_deg}
\end{figure}

\end{document}